\documentclass[11pt]{article} 
\usepackage{amssymb, amsmath, amsthm, srcltx} 
\usepackage{graphicx} 
\usepackage{color} 
\newtheorem{thm}{Theorem}[section] 
\newtheorem{lem}{Lemma}[section] 
\newtheorem{cor}[lem]{Corollary} 
 
\theoremstyle{definition} 
 
\theoremstyle{remark} 
 
\numberwithin{equation}{section} 

\newcommand{\E}{\mathbf{E}\,} 
 
\newcommand{\Tr}{\mathrm{Tr}\;\!} 
 
\newcommand{\im}{\mathrm{Im}\;\!}

\newenvironment{Proof of}{\removelastskip\par\medskip 
\noindent{\em Proof of} \rm}{\penalty-20\null\hfill$\square$\par\medbreak} 
 
\begin{document} 
\setcounter{page}{1} 
 
\title{\bf On the asymptotic distribution of the singular values of  powers of
random matrices  }

\author{{\bf N. Alexeev}\\{\small S.-Peterburg state University}\\{\small S.-Peterburg, Russia}\and{\bf F. G\"otze}\\{\small Faculty of Mathematics} 
\\{\small University of Bielefeld}\\{\small Germany} 
\and {\bf A. Tikhomirov}$^{1}$\\{\small Department of Mathematics 
}\\{\small Komi Research Center of Ural Branch of RAS} 
\\{\small Syktyvkar, Russia}} 
\maketitle 
 \footnote{$^1$Partially supported by RFBF 
grant N 09-01-12180. Partially supported RFBR--DFG, grant  N 09-01-91331. 
Partially supported by CRC 701 ``Spectral Structures and Topological 
Methods in Mathematics'', Bielefeld} 
 
 

\begin{abstract}We consider powers of  random  matrices  with independent entries.  Let $X_{ij},{  } i,j\ge 1$, be independent complex random variables 
with $\E X_{ij}=0$ and $\E |X_{ij}|^2=1$ and let
$\mathbf X$ denote  an $n\times n$  matrix with $[\mathbf X]_{ij}=X_{ij}$, for $1\le i, j\le n$. 
Denote by $s_1^{(m)}\ge\ldots\ge s_n^{(m)}$ the singular values  of 
the random  matrix 
$\mathbf W:={n^{-\frac m2}} \mathbf X^m$ 
and define  the empirical distribution of the squared singular values by 
$$ 
\mathcal F_n^{(m)}(x)=\frac1n\sum_{k=1}^nI_{\{{s_k^{(m)}}^2\le x\}}, 
$$ 
where $I_{\{B\}}$ denotes the indicator of an event $B$. 
We prove that under a Lindeberg condition for the  fourth moment that 
 the expected spectral distribution $F_n^{(m)}(x)=\E \mathcal F_n^{(m)}(x)$ converges to 
the  distribution function  $G^{(m)}(x)$ defined by its moments 
$$ 
\alpha_k(m):=\int_{\mathbb R}x^k\,d\,G(x)=\frac {1}{mk+1}\binom{km+k}{ k}. 
$$

\end{abstract} 
 
\maketitle 
\markboth{N. V. Alexeev, F. G\"otze, A.Tikhomirov}{Powers of random matrices} 
 
\section{Introduction} 
Let $X_{ij},{  } i,j\ge 1$, be independent complex random variables 
with $\E X_{ij}=0$ and $\E |X_{ij}|^2=1$ and 
$\mathbf X$ is an $n\times n$  matrix with $[\mathbf X]_{ij}=X_{ij}$, for $1\le i, j\le n$. 
Denote by $s_1^{(m)}\ge\ldots\ge s_n^{(m)}$ the singular values  of  the random  matrix 
$\mathbf W:={n^{-\frac m2}} \mathbf X^m$ 
and define the  empirical distribution of its squared singular values by 
$$ 
\mathcal F_n^{(m)}(x)=\frac1n\sum_{k=1}^n\mathbb I{\{{s_k^{(m)}}^2\le x\}}, 
$$ 
where $\mathbb I{\{B\}}$ denotes the indicator of an event $B$. 
We shall investigate the 
convergence of the expected spectral distribution $F_n^{(m)}(x)=\E \mathcal F_n^{(m)}(x)$ 
 to the  distribution function  $G^{(m)}(x)$ defined by its moments 
$$ 
\alpha_k(m):=\int_{\mathbb R}x^k\,d\,G(x)=\frac {1}{mk+1}\binom{km+k}{ k}. 
$$ 
The sequence ${{\alpha}_k(m)}$ consists of  the so-called  Fuss--Catalan 
Numbers. This sequence defines a distribution with Stieltjes transform $s^{(m)}(z)$ satisfying the equation (\ref{very}) below. 
We consider the  Kolmogorov distance between the distributions 
$F_n^{(m)}(x)$ and $G^{(m)}(x)$, that is 
$$ 
\Delta_n^{(m)}:=\sup_x| F_n^{(m)}(x)-G^{(m)}(x)|. 
$$ 
The main result of this paper is the following 
\begin{thm}\label{main}Let $\E X_{jk}=0$, $\E|X_{jk}|^2=1$, $\E |X_{jk}|^{4}\le M<\infty$. 
Assume that for any $\tau>0$ 
\begin{equation} 
 L_n(\tau):=\frac1{n^2}\sum_{j,k=1}^n\E|X_{jk}|^4I\{|X_{jk}|>\tau\sqrt n\}\to 0\quad\text{as}\quad n\to\infty, 
\end{equation} 
where $I\{E\}$ denotes indicator of an event $E$. 
Then, for any fixed $m\ge 2$, 
\begin{equation}\notag 
\lim_{n\to \infty}\sup_x| F_n^{(m)}(x)-G^{(m)}(x)| = 0. 
\end{equation} 
\end{thm} 
\begin{cor}\label{main1} 
 Let $X_{jk}$ are independent identically distributed complex random variables. Let 
\begin{equation}\label{AGT} 
 \E X_{jk}=0,\quad\E|X_{jk}|^2=1,\quad \E |X_{jk}|^{4}= M<\infty 
\end{equation} 
Then, for any fixed $m\ge 2$, 
\begin{equation}\notag 
\lim_{n\to \infty} \Delta_n^{(m)}= 0. 
\end{equation} 
\end{cor} 
 
Oravecz  in 2001, \cite{Oravecz}, studied the  so called $\mathcal
R$-elements introduced by Voiculescu and has shown that the $m$-th powers of
these elements have a distribution whose  moments are Fuss--Catalan numbers. 
These numbers satisfy the following simple recurrence relation 
\begin{equation}\label{recur} 
\alpha_k(m)=\sum_{k_0+\cdots+k_m=k-1}\prod_{\nu=0}^m\alpha_{k_{\nu}}(m). 
\end{equation} 
Denote by $s^{(m)}(z)$ the Stieltjes transform of the distribution with
moments $\alpha_k(m)$. Using equality (\ref{recur}), we may show that the Stieltjes transform 
$s^{(m)}(z)$ satisfies the equation 
\begin{equation}\label{very} 
1+zs^{(m)}(z)+(-1)^{m+1}z^m(s^{(m)}(z))^{m+1}=0. 
\end{equation} 
Distributions with such a Stieltjes transform belong to the class of the
so-called Free Bessel Laws which are described in Banica and others
\cite{capitaine:08}. This distribution has been studied  also in
\cite{Speicher}. Using Free probability theory it is possible to prove the
result of Theorem \ref{main} for random matrices with independent entries,
provided that   all moments of $\mathbf X$ are finite. 
See for instance, Mingo and Speicher  \cite{Speicher}, and T. Banica and others \cite{capitaine:08}. 
Theorem  \ref{main} was formulate in \cite{AGT:2010a}.
In \cite {AGT} we gave a proof of  Theorem \ref{main} by the  method of moments. 
Here we present a proof of  Theorem \ref{main} using Stieltjes transforms. 
This approach allows us to get some bound of the rate of convergence. 
Our proof of Theorem \ref{main} is based on the representation (\ref{very}). 
 We shall investigate the  Stieltjes transform $s_n^{(m)}(z)$ of the distribution function $F_n^{(m)}(x)$ and 
 we shall show that $s_n^{(m)}(z)$ satisfies an equation 
\begin{equation}\notag 
1+zs_n^{(m)}(z)+(-1)^{m+1}z^m(s_n^{(m)}(z))^{m+1}=\delta_n(z) 
\end{equation} 
with some function $\delta_n(z)\to 0$ as $n\to\infty$. 
From these two relations we get that 
$s_n^{(m)}(z)$ converges  to $s^{(m)}(z)$ uniformly on any compact set in the
upper half-plane $\mathcal K\subset \mathcal C^+$. The last claim is
equivalent to weak convergence of the distribution functions 
$F_n^{(m)}(x)$ to the distribution function $F^{(m)}(x)$. 
%

\section{Auxiliary results} 
In this Section we describe a 
 symmetrization of one-sided distributions and  a special 
representation of the symmetrizing distribution of squared singular 
values of random matrices. Furthermore,  we 
shall  modify the random matrix $X$ by truncation of its entries. 
By condition (\ref{AGT}),we get  $\tau^{-q}L_n(\tau)\to0)$, for any $\tau>0$
and that for
any  $q>0$ the  function $\tau^{-q}L_n(\tau)$ is not increasing in $\tau$. 
This implies that  we may choose a sequence of positive numbers $\tau_n>0$, $n=1,2\ldots$ such that 
\begin{equation} 
 \tau_n\to0, \quad\text{and}\quad L_n(\tau)\le\tau_n^6,\quad\text{as}\quad n\to\infty. 
\end{equation} 
\subsection{Truncation} 
We call the matrix  $\widetilde{\mathbf X}$  the truncation of 
  $\mathbf X$ if 
\begin{equation} 
\widetilde{X}_{ij} = 
\begin{cases} 
{X}_{ij},\text{ if } |X_{ij}|<\tau_n \sqrt{n}\\ 
0,\text{ otherwise} 
\end{cases}. 
\end{equation} 
Denote by $\widetilde s_1^{(m)}\ge\ldots\ge \widetilde s_n^{(m)}$ the singular values  of 
the random  matrix 
$ 
\widetilde{\mathbf W}:= n^{-\frac m2}{\widetilde{\mathbf X}}^m 
$ 
and define the empirical distribution of its squared singular values 
(eigenvalues of the matrix $\widetilde{\mathbf V}=\widetilde{\mathbf W}{\widetilde{\mathbf W}}^*$) by 
$ 
{\widetilde{\mathcal F}_n}^{(m)}(x)= 
\frac1n\sum_{k=1}^nI\{({\widetilde s}_k^{(m)})^2\le x\} 
$. 
Let $\widetilde F_n^{(m)}(x)=\E {\widetilde{\mathcal F}_n}^{(m)}(x)$. It is straightforward to check that 
\begin{align}\notag 
 \sup_x|{\widetilde F}_n^{(m)}(x)-F_n^{(m)}(x)|&\le \sum_{j,k=1}^n\Pr\{|X_{jk}|\ge cn^{\frac12}\}\notag\\&\le 
\frac 1{n^2\tau_n^4}\sum_{j,k=1}^n\E|X_{jk}|^{4}I\{|X_{jk}|>\tau_n \sqrt{n}\}=\frac{L_n(\tau_n)}{\tau_n^4}\le\tau_n^2. 
\end{align} 
Introduce the matrices $\widehat{\mathbf X}:=\widetilde{\mathbf X}-\E\widetilde{\mathbf X}$ 
and $\widehat{\mathbf W}={\widehat{\mathbf X}}^m$ and $\widehat{\mathbf
  V}=\widehat{\mathbf W}{\widehat{\mathbf W}}^*$. Let ${\widehat {\mathcal
    F}_n}^{(m)}$ by the   empirical distribution of its eigenvalues 
(squared singular values of $\widehat{\mathbf W}$) and ${\widehat F_n}^{(m)}= 
\E{\widehat {\mathcal F}_n}^{(m)}$. Let $\widehat s_n^{(m)}(z)$ denote the
Stieltjes transform of 
${\widehat F_n}^{(m)}$. 
Introduce the resolvent matrices 
\begin{equation} 
\widehat{\mathbf R}=(\widehat{\mathbf V}-z\mathbf I)^{-1},\qquad\text{and}\qquad 
{\widetilde{\mathbf R}}=(\widetilde{\mathbf V}-z\mathbf I)^{-1}. 
\end{equation} 
We have 
\begin{equation}\widetilde s_n^{(m)}(z)=\frac1n\E\Tr{\widetilde{\mathbf R}}\quad\text{and}\quad 
 \widehat s_n(z)=\frac1n\E\Tr{\widehat{\mathbf R}}. 
\end{equation} 
Applying the resolvent equality 
\begin{equation} 
 (\mathbf A+\mathbf B-z\mathbf I)^{-1}=(\mathbf A-z\mathbf I)^{-1}-(\mathbf A-z\mathbf I)^{-1} 
\mathbf B(\mathbf A+\mathbf B-z\mathbf I)^{-1}, 
\end{equation} 
 we get 
\begin{equation}\label{resolv} 
 |\widetilde s_n^{(m)}(z)-\widehat s_n^{(m)}(z)|\le \frac1{n}\E|\Tr {\widetilde{\mathbf R}}(\widetilde{\mathbf V}-{\widehat{\mathbf V}})\widehat{\mathbf 
 R}|. 
\end{equation} 
Using that $\Tr {\widetilde{\mathbf R}}(\widetilde{\mathbf V}-{\widehat{\mathbf V}})\widehat{\mathbf 
 R}=\Tr {(\widetilde{\mathbf V}-{\widehat{\mathbf V}})\widehat{\mathbf 
 R}\widetilde{\mathbf R}}$ and applying H\"older's inequality, we obtain 
\begin{equation}\label{fin100} 
 |{\widetilde s}_n^{(m)}(z)-\widehat s_n^{(m)}(z)|\le \frac1{\sqrt n v^2}\E^{\frac12}\|{\widehat{\mathbf W}}-{\widetilde{\mathbf W}}\|_2^2(\E^{\frac12}\|\widetilde{\mathbf W}\|^2+\E^{\frac12}\|\widehat{\mathbf W}\|^2). 
\end{equation} 
By definition of the  matrices $\widetilde {\mathbf W}$ and $\widehat{\mathbf W}$, we get 
\begin{equation} 
 \widetilde {\mathbf W}-\widehat{\mathbf W}=\sum_{\nu=0}^{m-1}{\widehat{\mathbf X}}^{\nu} 
(\widetilde {\mathbf X}-\widehat{\mathbf X}){\widehat{\mathbf X}}^{m-1-\nu}. 
\end{equation} 
This implies that 
\begin{equation} 
 \|\widehat{\mathbf W}-\widetilde{\mathbf W}\|_2^2\le m\sum_{\nu=0}^{m-1}\|\widetilde{\mathbf X}-\widehat{\mathbf X}\|_2^2\|{\widehat{\mathbf X}}^{\nu}{\widetilde{\mathbf X}}^{m-1-\nu}\|_2^2. 
\end{equation} 
Applying Lemma \ref{norm0}, we obtain 
\begin{equation}\label{fin10} 
 \E\|\widehat{\mathbf W}-\widetilde{\mathbf W}\|_2^2\le Cn\|\E\widetilde{\mathbf X}\|_2^2\le 
\frac {CM}{n\tau_n^6}L_n(\tau_n)\le Cn^{-1}. 
\end{equation} 
Inequalities (\ref{fin10}) and (\ref{fin100}) together imply 
\begin{equation} 
 |\widetilde s_n^{(m)}(z)-\widehat s_n^{(m)}(z)|\le \frac {C}{\sqrt nv^2}. 
\end{equation} 
 
Now we conclude that 
\begin{equation}\notag 
 \lim_{n\to\infty}\sup_x|F_n^{(m)}-G^{(m)}(x)|=\lim_{n\to\infty} \sup_x|\E \widehat F_n^{(m)}-G^{(m)}(x)|. 
\end{equation} 
In the what follows we may assume without lost of generality that 
\begin{equation}\label{conditions} 
 \E X_{jk}=0,\quad \E |X_{jk}|^2=1, \quad\text{and}\quad |X_{jk}|\le \tau_n\sqrt n, 
\end{equation} 
for some $\tau_n>0$ such that $\tau_n\to 0$, $L_n(\tau_n)\le \tau_n^6$  as $n\to\infty$. 
 
\subsection{Symmetrization} 
We shall use the following  ``symmetrization'' of one-sided distributions. Let $\xi^2$ 
 be a positive random variable with distribution function $F(x)$. Define 
$\widetilde \xi:=\varepsilon\xi$ where $\varepsilon$   a Rademacher random
variable with 
 $\Pr\{\varepsilon=\pm1\}=1/2$ which is independent of $\xi$ . Let $\widetilde F(x)$ denote 
the distribution function of $\widetilde \xi$. It satisfies the 
equation 
\begin{equation}\label{sym} 
\widetilde F(x)=1/2(1+\text{sgn} \{x\}\,F(x^2)), 
\end{equation} 
We shall apply this symmetrization to the distribution of the squared 
singular values of the matrix $\mathbf W$. Introduce the following matrices 
\begin{align}\notag 
\mathbf V=\left(\begin{matrix}{\mathbf W\quad \mathbf O}\\{\mathbf O\quad 
\mathbf W^*}\end{matrix}\right),\quad 
\mathbf J=\left(\begin{matrix}{\mathbf O\quad \mathbf I_{n}}\\{\mathbf I_{n}\quad 
\mathbf O}\end{matrix}\right), 
\quad\text{and}\quad 
\widehat {\mathbf V}=\mathbf V\mathbf J\notag 
\end{align} 
Here and in the what follows  $\mathbf A^*$ denotes the adjoined (transposed
and complex conjugate)  matrix $\mathbf A$ and $\mathbf I_k$ denotes the unit matrix of order $k$. 
Note that $\widehat{\mathbf V}$ is a Hermitian matrix. The eigenvalues of the matrix $\widehat{\mathbf V}$ are $-s_1,\ldots,-s_n,s_n,\ldots,s_1$. 
Note that the symmetrization of the distribution function $\mathcal 
F_n(x)$ is a function $\widetilde{\mathcal F}_n(x)$ which is the empirical  
distribution function of the eigenvalues of  the matrix $\widehat {\mathbf
  V}$. 
By (\ref{sym}), we have 
\begin{equation}\label{sym10} 
\Delta_n^{(m)}=\sup_x|\widetilde F_n^{(m)}(x)-\widetilde 
G^{(m)}(x)|, 
\end{equation} 
where $\widetilde F_n^{(m)}(x)=\E\widetilde{\mathcal F}_n(x)$ and 
$\widetilde G^{(m)}(x)$ denotes the  symmetrization of the distribution function 
$G^{(m)}(x)$. 
Let $s(z)$ denote the  Stieltjes transform of the random variable $\xi^2$ and
let $\widetilde s(z)$ denote the Stieltjes transform of $\widetilde \xi$. Then 
\begin{equation}\label{sym12} 
\widetilde s(z)=zs(z^2). 
\end{equation} 
Equations (\ref{very}) and (\ref{sym12}) together imply 
\begin{equation}\label{very1} 
 1+z\widetilde s^{(m)}(z)+(-1)^{m+1}z^{m-1}(\widetilde s^{(m)}(z))^{m+1}=0. 
\end{equation} 
 
In the what follows we shall consider the symmetrization of the distribution
$F_n^{(m)}(x)$ 
and the Stieltjes transform of ${\widetilde F}_n^{(m)}(x)$. 
We shall omit ``$\quad\widetilde{}\quad$'' in the notation of the distribution function 
${\widetilde F}_n^{(m)}(x)$ (${\widetilde G}^{(m)}(x)$) 
and the Stieltjes transform 
${\widetilde s}_n^{(m)}(za)$ (${\widetilde s}^{(m)}(z)$). By 
$C$ (with an index or without it) we shall denote generic 
absolute constants, 
whereas $C(\,\cdot\,,\,\cdot\,)$ will denote 
positive constants depending on arguments. For every matrix $\mathbf A$ by
$\|\mathbf A\|_2$ 
we shall denote the Hilbert--Schmidt norm of the matrix $\mathbf A$  and by
$\|\mathbf A\|$ we  shall denote the  operator norm of the matrix $\mathbf A$.

\section{  The proof of the main result for $m=2$} 
First, we prove Theorem \ref{main} for $m=2$. 
Introduce the matrices $\mathbf H$ and $\mathbf J$ by the equalities 
\begin{equation} 
\mathbf H=\left(\begin{matrix}{\mathbf X\quad \mathbf O}\\{\mathbf O\quad \mathbf X^*}\end{matrix}\right), 
\qquad 
\mathbf { J}:=\left(\begin{matrix}{\mathbf O\quad \mathbf I}\\{\mathbf I\quad \mathbf O}\end{matrix}\right). 
\end{equation} 
Let $\mathbf V:=\mathbf H^m\mathbf {{J}}$, and $\mathbf R(z)$ denote
the  resolvent matrix of  $\mathbf V$, 
\begin{equation}\notag 
\mathbf R(z):=(\mathbf V-z\mathbf I)^{-1}. 
\notag 
\end{equation} 
 
Furthermore, we note that the symmetrization of the distribution function
$G_2(x)$ 
has a Stieltjes transform $s(z)$ which satisfies the following equation 
\begin{equation}\label{main211} 
1+zs(z)-zs^{3}(z)=0. 
\end{equation} 
We shall prove that in the case $m=2$ the Stieltjes transform of the
expected spectral distribution function 
$s_n(z)=\int_{-\infty}^{\infty}\frac1{x-z} 
\text{d} F_n(x)$ satisfies the equation 
\begin{equation}\label{main3} 
1+zs(z)-zs^3(z)=\delta_n(z), 
\end{equation} 
where $\delta_n(z)$ denotes some function such that $\delta_n(z)\to 0$ as $n\to \infty$. 
In the what follows we shall denote by 
$\varepsilon_n(z)$ a generic error term 
such that $|\varepsilon_n(z)|\le C\tau_n^av^{-b}$ for some positive constants $C,a$, and $b$. 
 
We start from the obvious  equality 
\begin{equation}\label{main4} 
1+zs_n(z)=\frac1{2n}\E\Tr \mathbf V\mathbf R(z). 
\end{equation} 
Using the definition of the matrices $\mathbf V$, $\mathbf H$ and $\mathbf{ J}$, we get 
\begin{equation}\label{auxiliary0} 
1+zs_n(z)=\frac1{2n\sqrt n}\sum_{j,k=1}^n\E X_{jk}\left(\left[ 
\mathbf {H}\mathbf J\mathbf R\right]_{kj}+\left[\mathbf {H}\mathbf J\mathbf R\right]_{j+n,k+n}\right). 
\end{equation} 
By Lemma \ref{teilor} of the Appendix, we get 
\begin{equation}\notag 
1+zs_n(z))=\frac1{2n\sqrt n}\sum_{j,k=1}^n\E\left[\frac{\partial \mathbf {H}\mathbf J\mathbf R}{\partial X_{jk}}\right]_{kj}+ 
\frac1{2n\sqrt n}\sum_{j,k=1}^n\E\left[\frac{\partial \mathbf {H}\mathbf J\mathbf R}{\partial X_{jk}}\right]_{j+n,k+n}+\varepsilon_n(z). 
\end{equation} 
 
 
 
Let $\mathbf e_1,\ldots,\mathbf e_{2n}$ be an orthonormal basic of $\mathbb 
R^{2n}$. First we note that, for $1\le j,k\le n$, 
\begin{equation}\label{auxiliary2} 
\frac{\partial  \mathbf H}{\partial X_{jk}}=\frac1{\sqrt n}(\mathbf e_j\mathbf e_k^T+\mathbf e_{k+n}\mathbf e_{j+n}^T), 
\end{equation} 
and 
\begin{equation}\label{auxiliary3} 
\frac{\partial  (\mathbf H\mathbf{ J})}{\partial X_{jk}}=\frac1{\sqrt n}(\mathbf e_j\mathbf e_{k+n}^T+\mathbf e_{k+n}\mathbf e_{j}^T). 
\end{equation} 
Now we compute  the derivatives of the resolvent matrix as follows 
\begin{align}\label{auxiliary4} 
\frac{\partial \mathbf R}{\partial X_{jk}}= 
&-\frac1{\sqrt n }\mathbf R(\mathbf e_j\mathbf e_k^T+\mathbf e_{k+n}\mathbf e_{j+n}^T)\mathbf H\mathbf{ J}\mathbf R\notag\\ 
&-\frac1{\sqrt n} \mathbf R\mathbf{H}(\mathbf e_j\mathbf e_{k+n}^T+\mathbf e_{k+n}\mathbf e_{j}^T)\mathbf R. 
\end{align} 
and 
\begin{align}\label{auxiliary5} 
\frac{\partial (\mathbf H\mathbf {\mathbf J}\mathbf R)}{\partial X_{jk}}=\frac1{\sqrt 
n} (\mathbf e_j\mathbf e_{k+n}^T+\mathbf e_{k+n}\mathbf e_{j}^T)\mathbf R 
&-\frac1{\sqrt n }\mathbf H\mathbf J\mathbf R(\mathbf e_j\mathbf e_k^T+\mathbf e_{k+n}\mathbf e_{j+n}^T)\mathbf H\mathbf J\mathbf R\notag\\ 
&-\frac1{\sqrt n} \mathbf H\mathbf J\mathbf R\mathbf{H} 
(\mathbf e_j\mathbf e_{k+n}^T+\mathbf e_{k+n}\mathbf e_{j}^T)\mathbf R. 
\end{align}


The equalities (\ref{auxiliary0}) and (\ref{auxiliary5}) together imply 
\begin{equation}\label{auxiliary6} 
1+zs_n(z)=A_1+\cdots+A_6+\varepsilon_n(z), 
\end{equation} 
where 
\begin{align}\notag 
A_1&:=\frac1{2n^2}\E(\sum_{j=1}^n \mathbf R_{j,j+n}+\sum_{j=1}^n \mathbf R_{j+n,j}), 
\notag\\ 
A_2&:=-\frac1{2n^2}\E\sum_{j,k=1}^n ([\mathbf HJ\mathbf R]_{jk}^2 
+[\mathbf H\mathbf J\mathbf R]_{j+n,k+n}^2),\notag\\ 
A_3&:=-\frac1{n^2}\E\sum_{j=1}^n [\mathbf H\mathbf J\mathbf R]_{j,j+n} 
\sum_{k=1}^n [\mathbf H\mathbf J\mathbf R]_{k+n,k},\notag\\ 
A_4&:=-\frac1{2n^2}\E\sum_{j,k=1}^n ([\mathbf H\mathbf J\mathbf R\mathbf H]_{k,j+n}\mathbf R_{k+n,j} 
+[\mathbf H\mathbf J\mathbf R\mathbf H]_{j+n,k+n}\mathbf R_{j,k+n}),\notag\\ 
A_5&:=-\frac1{2n^2}\E\sum_{k=1}^n [\mathbf H\mathbf J\mathbf R\mathbf H]_{k,k+n} 
\sum_{j=1}^n\mathbf R_{jj},\notag\\ 
A_6&:=-\frac1{2n^2}\E\sum_{k=1}^n [\mathbf H\mathbf J\mathbf R\mathbf H]_{k+n,k} 
\sum_{j=1}^n\mathbf R_{j+n,j+n}.\notag 
\end{align} 
We prove that the first four summands are negligible and the main  
asymptotic terms are the last two summands. We now start  the investigation 
of these terms. 
 
 
\begin{lem}\label{lem1}Under conditions of Theorem \ref{main} we have 
\begin{align}\label{l1} 
|A_5&+(\frac1{2n}\sum_{k=1}^n \E[\mathbf H \mathbf J\mathbf R\mathbf H]_{k,k+n})(\frac1n\sum_{j=1}^n\E \mathbf R_{jj})|\le \frac C{nv^2},\notag\\ 
|A_6&+(\frac1{2n}\sum_{k=1}^n \E[\mathbf H\mathbf J 
\mathbf R\mathbf H]_{k+n,k})(\frac1n\sum_{j=1}^n\E \mathbf R_{j+n,j+n})|\le \frac C{nv^2}. 
\end{align} 
\end{lem} 
\begin{proof}Using  Cauchy's inequality we have 
\begin{align} 
|A_5&+(\frac1{2n}\sum_{k=1}^n \E[\mathbf H\mathbf{J}\mathbf R\mathbf H]_{k,k+n})(\frac1n\sum_{j=1}^n\E \mathbf R_{jj})|\\ \le &\E^{\frac12}\left|\frac1n(\sum_{k=1}^n [\mathbf H\mathbf{J} \mathbf R\mathbf H]_{k,k+n}- 
\E\sum_{k=1}^n [\mathbf H\mathbf{J} \mathbf R\mathbf H]_{k,k+n})\right|^2\E^{\frac12}\left|\frac1n(\sum_{j=1}^n(\mathbf R_{jj}-\E \mathbf R_{jj}))\right|^2. 
\end{align} 
Applying Lemma \ref{var1} with $p=2$ and $q=1$ and Lemma \ref{var0} 
(see the Appendix),  we get 
\begin{equation} 
 |A_5+(\frac1{2n}\sum_{k=1}^n \E[\mathbf H\mathbf{J}\mathbf R\mathbf H]_{k,k+n})(\frac1n\sum_{j=1}^n\E \mathbf R_{jj})|\le\frac C{nv^2}. 
\end{equation} 
 
Similar we prove the second inequality in (\ref{l1}). 
Thus the Lemma is proved. 
\end{proof} 
Note that 
\begin{equation}\label{auxiliary7} 
\frac1n\sum_{j=1}^n\E \mathbf R_{jj}=\frac1n\sum_{j=1}^n\E \mathbf R_{j+n,j+n}=s_n(z). 
\end{equation} 
Lemma \ref{lem1}, equality (\ref{auxiliary7}) and the definition of matrix ${\mathbf H}$ together imply 
\begin{equation} 
A_5=-s_n(z)\frac1{2n}\sum_{j,k=1}^n\E X_{jk}\, [\mathbf H\mathbf{J}\mathbf R]_{j,k+n}+\frac{C\theta}{n v^4}, 
\end{equation} 
and similarly 
\begin{equation} 
A_6=-s_n(z)\frac1{2n}\sum_{j,k=1}^n\E X_{jk}\,  [\mathbf H\mathbf{J}\mathbf R]_{k+n,j}+\frac{C\theta}{n v^4}, 
\end{equation} 
where $\theta$ denotes a quantity such that $|\theta|\le 1$. 
Applying Lemma \ref{teilor} and equalities (\ref{auxiliary2})--(\ref{auxiliary5}), we get 
\begin{align}\label{auxialiry8} 
A_5&=-\frac12 s_n^2(z)+s_n^2(z)\frac1{2n}\sum_{j=1}^n\E [\mathbf H^2\mathbf{J}\mathbf R]_{jj}+A_7+A_9,\notag\\ 
A_6&=-\frac12 s_n^2(z)+s_n^2(z)\frac1{2n}\sum_{j=1}^n\E [\mathbf H^2\mathbf{J}\mathbf R]_{j+n,j+n}+A_8+A_{10}, 
\end{align} 
where 
\begin{align} 
A_7&=s_n(z)\frac1{2n^2}\sum_{j=1}^n[\mathbf H\mathbf{J}\mathbf R]_{j+n,j+n}\sum_{k=1}^n[\mathbf H\mathbf{J}\mathbf R]_{k+n,k},\notag\\ 
A_8&=s_n(z)\frac1{2n^2}\sum_{j=1}^n[\mathbf H\mathbf{J}\mathbf R]_{j,j+n}\sum_{k=1}^n[\mathbf H\mathbf{J}\mathbf R]_{k,k},\notag\\ 
A_9&=s_n(z)\frac1{2n^2}\sum_{j,k=1}^n[\mathbf H\mathbf{J}\mathbf R\mathbf H]_{j+n,k+n}\mathbf R_{j,k+n},\notag\\ 
A_{10}&=s_n(z)\frac1{2n^2}\sum_{j,k=1}^n[\mathbf H\mathbf{J}\mathbf R\mathbf H]_{j+n,k+n}\mathbf R_{j,k+n},\notag 
\end{align} 
By resolvent equality $\mathbf I+z\mathbf R=\mathbf W\mathbf R$, we have 
\begin{equation}\label{in101} 
\frac1{2n}(\sum_{j=1}^n\E [\mathbf H^2\mathbf{J}\mathbf R]_{jj}+\sum_{j=1}^n\E [\mathbf H^2\mathbf{J}\mathbf R]_{j+n,j+n})=1+zs_n(z). 
\end{equation} 
Equalities (\ref{main4}), (\ref{auxialiry8}) and (\ref{in101}) together imply 
\begin{equation}\label{in10} 
A_5+A_6=zs_n^3(z)+A_7+\cdots+A_{10}. 
\end{equation} 
 

\begin{lem}\label{lem6}Under the conditions of Theorem \ref{main} we have 
\begin{equation}\label{in12} 
\max\{|A_1|,\,|A_2|,\,|A_9|,\,|A_{10}|,\,|A_4|\}\le \frac C{nv^2}. 
\end{equation} 
\end{lem} 
\begin{proof} We shall describe the  estimate  (\ref{in12}) for the quantity $A_{9}$
  only. The other bounds will be similar. 
By H\"older's inequality, we have 
\begin{equation}\notag 
|A_{9}|\le \frac1{n^2}\E\|{\mathbf H}\mathbf J\mathbf R\mathbf H\|_2\|\mathbf R\|_2\le \frac1{n^{\frac32}v}\E\|{\mathbf H}\mathbf J\mathbf R\mathbf H\|_2\|, 
\end{equation} 
where $\|\cdot\|_2$ denotes the Hilbert--Schmidt norm of a matrix. Using 
\begin{equation} 
\|{\mathbf H}\mathbf{J}\mathbf R\mathbf H\|_2=\|\mathbf H^2\mathbf{J}\mathbf R\|_2\le\|\mathbf H^2\|_2\|\mathbf R\|\le \frac1v\|\mathbf H^2\|_2, 
\end{equation} 
and Lemma \ref{norm0}, we get 
\begin{equation}\notag 
|A_9|\le \frac{C\sqrt n}{n^{\frac32}v^2}\le \frac {C}{nv^2} 
\end{equation} 
Thus the Lemma is proved. 
\end{proof} 
 
Introduce the notations 
\begin{align} 
A:&=\frac1n\sum_{j=1}^n\E[\mathbf H\mathbf{J}\mathbf R]_{jj},\quad B:=\frac1n\sum_{j=1}^n\E[\mathbf H\mathbf{J}\mathbf R]_{j+n,j+n},\notag\\ 
C:&=\frac1n\sum_{j=1}^n\E[\mathbf H\mathbf{J}\mathbf R]_{j,j+n},\quad D:=\frac1n\sum_{j=1}^n\E[\mathbf H\mathbf{J}\mathbf R]_{j+n,j},\quad t(z):=\sum_{j=1}^n \E \mathbf R_{j,j+n}.\notag 
\end{align} 
Using these  notations we prove the following 
\begin{lem}\label{rel}The following representations hold 
\begin{align} 
A&=-s_n(z)C-s_n(z)D+\varepsilon_n(z),\quad 
B=-s_n(z)D-s_n(z)C+\varepsilon_n(z),\notag\\ 
C&=-t_n(z)D-s_n(z)A+\varepsilon_n(z),\quad D=-t_n(z)C-s_n(z)B+\varepsilon_n(z),\notag 
\end{align} 
where $|\varepsilon_n(z)|\le \frac {C\tau_n}{nv^4}.$ 
\end{lem} 
\begin{proof}We start with the first equality. By definition of $A$, we have 
\begin{equation}\notag 
A=\frac1{n\sqrt n}\sum_{j,k=1}^n \E X_{jk}\mathbf R_{k+n,j}. 
\end{equation} 
Using Lemma \ref{teilor}, we get 
\begin{align} 
A&=-\frac1{n^2}\sum_{j,k=1}^n \E \mathbf R_{j+n,j+n}[\mathbf H\mathbf{J}\mathbf R]_{k+n,k}\notag\\&-\frac1{n^2}\sum_{j,k=1}^n \E [\mathbf R\mathbf H]_{j+n,j+n} 
[\mathbf{J}\mathbf R]_{k+n,k}-\frac1{n^2}\sum_{j,k=1}^n \E \mathbf R_{j+n,k}[\mathbf H\mathbf{J}\mathbf R]_{j+n,k} 
\notag\\&-\frac1{n^2}\sum_{j,k=1}^n \E [\mathbf R\mathbf H]_{k+n,j}[\mathbf{J}\mathbf R]_{k,j}+\varepsilon_n(z).\notag 
\end{align} 
Applying Lemma \ref{var1} and \ref{var2}, we have
\begin{equation}\notag 
A=-s_n(z)C-s_n(z)D+\varepsilon_n(z). 
\end{equation} 
The proof of the other relations is similar. 
\end{proof} 
We may write now 
\begin{equation}\notag 
A_3+A_7+A_8=-CD+\frac12s_n(z)(BD+AC)=-\frac12(C-s_n(z)B)D-\frac12(D-s_n(z)A)C+\varepsilon_n(z). 
\end{equation} 
Applying the results of Lemma \ref{rel}, we obtain 
\begin{align} 
A_3+A_7+A_8&=-\frac12(D+C)^2-\frac12t_n(z)CD+\varepsilon_n(z)\notag\\&=-\frac12(1+\frac{t_n}4)(D+C)^2-\frac18t_n(z)(C-D)^2+\varepsilon_n(z)\notag\\ 
&=-\frac{\varepsilon_{n}^2(z)(1+t_n(z)/4)^2}{2(1+t_n(z)-2s_n^2(z))^2}-\frac{\varepsilon_{n}^2(z)}{(1-t_n(z))^2}+\varepsilon_{n}(z).\notag 
\end{align} 
Consider first the case $v\ge 4$. Here we have 
\begin{equation}\notag 
|s_n(z)|\le \frac14,\quad |t_n(z)|\le \frac14. 
\end{equation} 
These inequalities imply that for $v\ge 4$ 
\begin{equation}\label{in000} 
|A_3+A_7+A_8|\le \frac {C\tau_n}{v^4}. 
\end{equation} 
Inequalities (\ref{auxiliary6}), (\ref{in10}), (\ref{in12}), and (\ref{in000}) together imply 
\begin{equation}\label{direct} 
1+zs_n(z)=zs_n^3(z)+\delta_n(z), 
\end{equation} 
where $|\delta_n(z)|\le \frac {C\tau_n}{v^4}$, for $v>4$.

\begin{lem}\label{stieltjes} Assuming the conditions of Theorem \ref{main} 
there exists some positive constants $C_0, \,C_1$ such that, for $v\ge C_0$, 
\begin{equation} 
|s(z)-s_n(z)|\le \frac{C_1|\varepsilon_n(z)|}{v}. 
\end{equation} 
 
\end{lem} 
\begin{proof}First we note that 
\begin{equation} 
 |zs_n(z)|\le 1+\frac1v\E^{\frac12}\|\mathbf V\|_2^2. 
\end{equation} 
Applying Lemma \ref{norm0} and that $\max\{|s(z)|,\,|s_n(z)\}\le v^{-1}$, we get 
\begin{equation}\label{in30} 
\max\{|zs_n^2(z)|,\,|zs(z)s_n(z)|\}\le \frac1v(1+\frac {C}v). 
\end{equation} 
Furthermore, 
\begin{equation}\label{in20} 
 \im\{zs^2(z)\}\le 0. 
\end{equation} 
It follows from equality (\ref{very}) that 
\begin{equation}\label{in17} 
\im zs^2(z)=\im\{z+\frac1{s(z)}\}=\frac{v|s(z)|^2-\im s(z)}{|s(z)|^2}. 
\end{equation} 
For a Stieltjes  transform $t(z)$ of a random variable $\xi$ we have 
\begin{equation}\label{in18} 
v|t(z)|^2-\im t(z)=v(\left|\E \frac1{\xi-z}\right|^2-\E\left|\frac1{\xi-z}\right|^2)\le 0. 
\end{equation} 
Equalities (\ref{in17}) and (\ref{in18}) together imply (\ref{in20}). 
From relations (\ref{very}) and (\ref{direct}) we obtain 
\begin{equation}\label{in50} 
|s_n(z)-s(z)|\le \frac{|\delta_n(z)|}{|z-zs^2(z)-zs(z)s_n(z)-zs^2_n(z)|}. 
\end{equation} 
Inequalities (\ref{in30}), (\ref{in20}) together imply that, for 
 $v\ge 4C$, 
\begin{equation}\label{in60} 
|z-zs^2(z)-zs(z)s_n(z)-zs^2_n(z)|\ge \im\{z-zs^2(z)-zs(z)s_n(z)-zs^2_n(z)\}\ge \frac v2. 
\end{equation} 
Inequalities (\ref{in50}) and (\ref{in60}) together completed the proof of lemma. 
\end{proof} 
The last Lemma implies that 
there exists an open set in $\mathcal C^+$  with non-empty interior 
such that $s_n(z)$ convergence to $s(z)$ on this set. 
 The Stieltjes transform of these
random variables is an  analytic function on $\mathcal C^+$ 
and locally bounded, that  is ($|s_n(z)|\le v^{-1}$ for any $v>0$).
 By Montel's Theorem (see, for instance, \cite{Conway}, p. 153, Theorem 2.9) 
$s_n(z)$ converges  to $s(z)$ uniformly on  any compact set 
in the upper half-plane $\mathcal K\subset \mathcal C^+$. 
  This implies that $\Delta_n\to 0$ as $n\to\infty$. 
Thus the proof of  Theorem \ref{main} in the case $m=2$ is complete.
 
\section{  The proof of the main result in general case}
Recall that $\mathbf H$ and $\mathbf J$ are defined by the equalities 
\begin{equation} 
\mathbf H=\left(\begin{matrix}{\mathbf X\quad \mathbf O}\\{\mathbf O\quad \mathbf X^*}\end{matrix}\right), 
\qquad 
\mathbf{ J}=\left(\begin{matrix}{\mathbf O\quad \mathbf I}\\{\mathbf I\quad \mathbf O}\end{matrix}\right). 
\end{equation} 
Let $\mathbf V:=\mathbf H^m\mathbf{J}$, and $\mathbf R(z)$ denote the
resolvent matrix of
the  matrix $\mathbf V$, 
\begin{equation}\notag 
\mathbf R(z):=(\mathbf V-z\mathbf I)^{-1}. 
\notag 
\end{equation} 
We shall use the following  ``symmetrization'' of a one-sided distribution.
 Let $\xi^2$ be a positive random variable. Define 
$\widetilde \xi:=\varepsilon\xi$,  where $\varepsilon$ denotes a Rademacher
random variable
 with $\Pr\{\varepsilon=\pm1\}=1/2$ which is independent of 
$\xi$. 
We apply this symmetrization to the distribution of the singular values of the
matrix $\mathbf X^2$. 
Note that the symmetrized  distribution function 
$\widetilde F_n(x)$ satisfies the equation 
\begin{equation}\notag 
\widetilde F_n(x)=1/2(1+\text{sgn} \{x\}\,F_n(x^2)), 
\end{equation} 
and that this function is the empirical spectral distribution function of 
the random matrix $$\mathbf V=\left(\begin{matrix}{\mathbf O\quad\mathbf X^m}\\ 
{\mathbf {X^*}^m\quad \mathbf O}\end{matrix}\right).$$ 
Furthermore, note that the symmetrization of the distribution function $G(x)$ 
has the Stieltjes transform $s(z)$ which satisfies the following equation 
\begin{equation}\label{main210} 
1+zs(z)+(-1)^{m+1}z^{m-1}s^{m+1}(z)=0. 
\end{equation} 
In the rest of paper we shall prove that the Stieltjes transform of 
the expected spectral distribution function $s_n(z)=\int_{-\infty}^{\infty}\frac1{x-z} 
\text{d}\E \widetilde F_n(x)$ satisfies the equation 
\begin{equation}\label{main31} 
1+zs_n(z)+(-1)^{m+1}z^{m-1}s_n^{m+1}(z)=\delta_n(z), 
\end{equation} 
where $\delta_n(z)$ denotes some remainder function 
such that $\delta_n(z)\to 0$ as $n\to \infty$. 
 
We start from the obvious  equality 
\begin{equation}\label{main41} 
1+zs_n(z)=\frac1{2n}\Tr \mathbf V\mathbf R(z). 
\end{equation} 
Using the definition of the matrices $\mathbf V$, $\mathbf H$ and $\mathbf {J}$, we get 
\begin{equation}\label{auxiliary01} 
1+zs_n(z)=\frac1{2n\sqrt n}\sum_{j,k=1}^n\E X_{jk}\left([\mathbf H^{m-1}\mathbf J\mathbf R]_{kj} 
+[\mathbf {H}^{m-1}\mathbf J\mathbf R]_{j+n,k+n}\right). 
\end{equation} 
In order to simplify the  calculations 
we shall assume that $X_{jk}$ are i.i.d. Gaussian random variables, 
and shall use the following well-known  equality for a Gaussian r.v. $\xi$ 
\begin{equation}\label{auxiliary11} 
\E\xi f(\xi)=\E f'(\xi), 
\end{equation} 
which holds for arbitrary differentiable functions $f(x)$, such that both
sides are defined. 
By Lemma \ref{teilor}, we obtain that the error of the replacement by Gaussian
r.v is of order $O(\tau_n)$. 
In the what follows we shall use the  notation $\varepsilon_n(z)$  
for  functions satisfying  $|\varepsilon_{n}(z)|\le 
C\tau_n^av^{-b}$, for some positive constants $a,b$, and $C$. 
Let $\mathbf e_1,\ldots,\mathbf e_{2n}$ denote an orthonormal basis  of $\mathbb R^{2n}$. 
First we note that 
\begin{equation}\label{auxiliary21} 
\frac{\partial  \mathbf H}{\partial X_{jk}}=\frac1{\sqrt n}(\mathbf e_j\mathbf e_k^T+\mathbf e_{k+n}\mathbf e_{j+n}^T). 
\end{equation} 
Now we may write the equality for the derivatives of the matrix $\mathbf H^{m-1}\mathbf J\mathbf R$ as follows 
\begin{align}\label{auxiliary41} 
\frac{\partial \mathbf H^{m-1}\mathbf J\mathbf R}{\partial X_{jk}}&=\frac1{\sqrt n}\sum_{q=0}^{m-2}\mathbf H^q(\mathbf e_j\mathbf e_k^T+\mathbf e_{k+n}\mathbf e_{j+n}^T) 
\mathbf H^{m-2-q}\mathbf J\mathbf R\notag\\ 
&-\frac1{\sqrt n }\sum_{q=0}^{m-1}\mathbf H^{m-1}\mathbf J\mathbf R\mathbf H^q(\mathbf e_j\mathbf e_k^T+\mathbf e_{k+n}\mathbf e_{j+n}^T)\mathbf H^{m-1-q}\mathbf J\mathbf R. 
\end{align} 
The equalities (\ref{auxiliary01}) and (\ref{auxiliary41}) together imply 
\begin{equation}\label{auxiliary61} 
1+zs_n(z)=A_1+A_2+B_1+B_2+C_1+C_2+D_1+D_2+\varepsilon_n(z), 
\end{equation} 
where 
\begin{align} 
A_1&:=\sum_{q=0}^{m-2}\frac1{2n^2}\E\sum_{j,k=1}^n\mathbf H^q_{kj}[\mathbf H^{m-2-q}\mathbf J\mathbf R]_{kj},\notag\\ 
A_2&=\sum_{q=0}^{m-2}\frac1{2n^2}\E\sum_{j,k=1}^n\mathbf H^q_{k,k+n}[\mathbf H^{m-2-q}\mathbf J\mathbf R]_{j+n,j},\notag\\ 
B_1&:=-\sum_{q=0}^{m-1}\frac1{2n^2}\E\sum_{j,k=1}^n [\mathbf H^{m-1}\mathbf J\mathbf R\mathbf H^{q}]_{k,j}[\mathbf H^{m-1-q}\mathbf J\mathbf R]_{k,j}\notag\\ 
B_2&=-\sum_{q=0}^{m-1}\sum_{k,j=1}^n \E[\mathbf H^{m-1}\mathbf J\mathbf R\mathbf H^{m-1-q}]_{k,k+n}[\mathbf H^{m-1-q}\mathbf J\mathbf R]_{j+n,j},\notag\\ 
C_1&=\sum_{q=0}^{m-2}\frac1{2n^2}\E\mathbf H^q_{j+n,k+n}[\mathbf H^{m-1-q}\mathbf J\mathbf R]_{j+n,k+n},\notag\\ 
C_2&:=\sum_{q=0}^{m-2}\frac1{2n^2}\E\sum_{j,k=1}^n\mathbf H^{q}_{j+n,j}[\mathbf H^{m-2-q}\mathbf J\mathbf R]_{k,k+n} 
\notag\\ 
D_1&:=-\sum_{q=0}^{m-1}\frac1{2n^2}\E\sum_{j,k=1}^n [\mathbf H^{m-1}\mathbf J\mathbf R\mathbf H^{q}]_{j+n,k+n}[\mathbf H^{m-1-q}\mathbf J\mathbf R]_{j+n,k+n}\notag\\ 
D_2&=-\sum_{q=0}^{m-1}\sum_{k,j=1}^n \E[\mathbf H^{m-1}\mathbf J\mathbf R\mathbf H^{m-1-q}]_{j+n,j}[\mathbf H^{m-1-q}\mathbf J\mathbf R]_{k,k+n}.\notag 
\end{align} 
\begin{lem} 
Under the conditions of Theorem \ref{main} there exists a constant $C>0$ that the following inequality holds 
\begin{equation} 
\max\{|A_1|,\,|B_1|\,|C_1|,\,|D_1|\}\le \frac C{nv}. 
\end{equation} 
\end{lem} 
\begin{proof} 
 To prove this lemma it is  enough to use H\"older's inequality and Lemma
 \ref{norm0} in the  Appendix. 
\end{proof} 
 
\begin{lem} 
Under the conditions of Theorem \ref{main} we have 
\begin{equation} 
A_2=C_2=0. 
\end{equation} 
\end{lem} 
\begin{proof} 
The claim follows immediately from the equality $\mathbf H^q_{j,j+n}=0$. 
 
\end{proof} 
 
To investigate the asymptotic behavior of $B_2$ and $D_2$ we introduce the notations 
\begin{align} 
f_{\alpha,\beta}&:=\frac1n\sum_{j=1}^n\E[\mathbf H^{\alpha}\mathbf J\mathbf R\mathbf H^{\beta}]_{j,j+n},\quad 
g_{\alpha,\beta}:=\frac1n\sum_{j=1}^n\E[\mathbf H^{\alpha}\mathbf J\mathbf R\mathbf H^{\beta}]_{j+n,j},\notag\\ 
t_{\alpha}&:=\frac1n\sum_{j=1}^n\E[\mathbf H^{\alpha}\mathbf J\mathbf R]_{jj},\quad 
u_{\alpha}:=\frac1n\sum_{j=1}^n\E[\mathbf H^{\alpha}\mathbf J\mathbf R]_{j+n,j+n}\notag 
\end{align} 
We prove the following 
\begin{lem}\label{lem11} Assuming the  conditions of Theorem \ref{main} 
there exists constant  $C>0$ such that the following inequality holds 
\begin{align} 
|B_2+\sum_{q=0}^{m-1}f_{m-1,q}g_{m-1-q,0}|\le\frac C{nv^4},\notag\\ 
|D_2+\sum_{q=0}^{m-1}g_{m-1,q}f_{m-1-q,0}|\le\frac C{nv^4}. 
\end{align} 
\end{lem} 
 
\begin{proof}Consider the first inequality. 
Applying H\"older's inequality, we get 
\begin{align} 
|B_2&+\sum_{q=0}^{m-1}f_{m-1,q}g_{m-1-q,0}|\notag\\&\le 
\sum_{q=0}^{m-1} 
\E^{\frac12}|\frac1n\sum_{j=1}^n ([\mathbf H^{m-1}\mathbf J\mathbf R\mathbf H^{q}]_{j,j+n}-\E\sum_{j=1}^n [\mathbf H^{m-1}\mathbf J\mathbf R\mathbf H^{q}]_{j,j+n})|^2 
\notag\\&\qquad\qquad\qquad\times\E^{\frac12}|\frac1n\sum_{j=1}^n ( [\mathbf H^{m-1-q}\mathbf J\mathbf R]_{j+n,j}-\E[\mathbf H^{m-1-q}\mathbf J\mathbf R]_{j+n,j})|^2. 
\notag 
\end{align} 
To conclude the proof of Lemma it is enough to use Lemmas \ref{var0} and \ref{var1}. 
The proof of the second inequality is similar. 
Thus the Lemma is proved. 
\end{proof} 
Note that 
\begin{equation}\label{auxiliary71} 
f_{00}=g_{00}=\frac1n\sum_{j=1}^n\E \mathbf R_{jj}=\frac1n\sum_{j=1}^n\E \mathbf R_{j+n,j+n}=s_n(z). 
\end{equation} 
By Lemma \ref{lem11} and  equality (\ref{auxiliary71}), we may write 
\begin{align}\label{a1} 
B_2+D_2&=-\frac12s_n(z)(f_{m-1,m-1}+g_{m-1,m-1})\notag\\&\qquad\qquad-\frac12\sum_{q=0}^{m-2}(f_{m-1,q}g_{m-1-q,0}+g_{m-1,q}f_{m-1-q,0}) +\varepsilon_n(z). 
\end{align} 
We consider now the behavior of the  coefficients $f_{\alpha,\beta}$,
$g_{\alpha,\beta}$, $t_{\alpha}$ and $u_{\alpha}$,
 for $\alpha,\beta=0,\ldots,m-1$. 
Applying Lemmas \ref{teilor} and \ref{var1}, we obtain the following relation for 
$\alpha>0$, $\beta>0$ 
\begin{align}\label{a2} 
 f_{\alpha,\beta}=-\sum_{q=0}^{m-1}f_{\alpha-1,q}t_{m-1+\beta-q}+f_{\alpha-1,\beta-1}+\varepsilon_n(z). 
\end{align} 
It is straightforward to check that for $q\ge m$ the following relation holds 
\begin{equation}\label{a3} 
 t_q=\frac 1n\sum_{j=1}^n[\mathbf H^q]_{j+n,j+n}+z\frac1n\sum_{j=1}^n[\mathbf H^{q-m}\mathbf R]_{j+n,j+n}=\delta_q+ 
zf_{q-m,0}+\varepsilon_n(z), 
\end{equation} 
where $\delta_0=1$ and $\delta_q=0$ for $q>0$. 
Using relation (\ref{a3}), we may rewrite (\ref{a2}) in the following form 
\begin{equation}\label{a4} 
f_{\alpha,\beta}=-z\sum_{q=0}^{\beta-1}f_{\alpha-1,q}f_{\beta-1-q,0} 
-\sum_{q=\beta}^{m-1}f_{\alpha-1,q}u_{m-1+\beta-q}+\varepsilon_n(z). 
\end{equation} 
For $\beta=0$, we get 
\begin{equation}\label{a5} 
f_{\alpha,0}=-\sum_{q=0}^{m-1}f_{\alpha-1,q}u_{m-1-q}+\varepsilon_n(z). 
\end{equation} 
Similar we obtain 
\begin{equation}\label{a6} 
g_{\alpha,\beta}=-z\sum_{q=0}^{\beta-1}g_{\alpha-1,q}g_{\beta-1-q,0} 
-\sum_{q=\alpha}^{m-1}g_{\alpha-1,q}t_{m-1+\beta-q}+\varepsilon_n(z). 
\end{equation} 
and 
\begin{equation}\label{a7} 
g_{\alpha,0}=-\sum_{q=0}^{m-1}g_{\alpha-1,q}t_{m-1-q}+\varepsilon_n(z). 
\end{equation} 
Applying Lemmas \ref{teilor} and \ref{var1}, we obtain a similar relation 
for $u_{\alpha}$ and $t_{\alpha}$, for $\alpha=0,\ldots,m-1$. 
\begin{equation}\label{a8} 
u_{\alpha}=-f_{\alpha-1,m-1}g_{0,0}-\sum_{q=0}^{m-2}f_{\alpha-1,q}g_{m-1-q,0}+\varepsilon_n(z), 
\end{equation} 
and 
\begin{equation}\label{a9} 
t_{\alpha}=-g_{\alpha-1,m-1}f_{0,0}-\sum_{q=0}^{m-2}g_{\alpha-1,q}f_{m-1-q,0}+\varepsilon_n(z). 
\end{equation} 
Denote by $\mathbf F$ (resp. $\mathbf G$) a $m-1\times m-1$ matrix with entries $F_{p,q}=f_{p-1,q-1}$ 
(resp. $\mathbf G_{p,q}=g_{p-1,q-1}$), $p,q=1\ldots,m$. 
Let $\mathbf t$ (resp. $\mathbf u$) denote a vector-column $(t_1,\ldots,t_{m-1})^T$ 
(resp $(u_1,\ldots, u_{m-1})^T$). Let $\mathbf f_{\alpha}=(f_{\alpha,0},,\ldots,f_{\alpha, \alpha-1},0,f_{\alpha, \alpha+1},\ldots,f_{\alpha, m-1})^T$ and 
$\mathbf g_{\alpha}=(g_{\alpha,0},\ldots,g_{\alpha, \alpha-1},0,g_{\alpha,
  \alpha+1},\ldots,g_{\alpha, m-1})^T$, for $\alpha=0,\ldots,m-2$. 
Introduce the  matrices 
\begin{equation}\notag 
 \mathbf M_u=\left(\begin{matrix}-u_0\quad\,-u_1\quad\ldots\quad-u_{m-3}\,-u_{m-2}\,\quad0 
\\-zf_{0,0}\,-u_0\quad\ldots\,\quad 
\,-u_{m-4}\,\quad0\,\,\quad-u_{m-2}\\ 
\ldots\\0\,-zf_{m-2,0}\,-zf_{m-3,0}\ldots\,-zf_{0,0}\,-u_0 
            \end{matrix}\right) 
\end{equation} 
and 
\begin{equation}\notag 
 \mathbf M_t=\left(\begin{matrix}-t_0\quad\,-t_1\quad\ldots\quad-t_{m-3}\,-t_{m-2}\,\quad0 
\\-zg_{0,0}\,-t_0\quad\ldots\,\quad 
\,-t_{m-4}\,\quad0\,\,\quad-t_{m-2}\\ 
\ldots\\0\,-zg_{m-2,0}\,-zg_{m-3,0}\ldots\,-zg_{0,0}\,-t_0 
            \end{matrix}\right) 
\end{equation} 
Let 
\begin{equation}\notag 
 \mathbf L=\left(\begin{matrix}0\,0 \ldots0\,1\\0\,0\ldots1\,0\\\ldots\\ 
1\,0\ldots0\,0 
               \end{matrix}\right) 
   \end{equation} 
We introduce as well the vectors $\mathbf y_{\alpha}=(-f_{0,\alpha-1},\ldots,-f_{0,1},0,-zf_{0,1},\ldots,-zf_{m-\alpha,0})^T$ and 
$\mathbf
w_{\alpha}=(-g_{0,\alpha-1},\ldots,-g_{0,1},0,-zg_{0,1},\ldots,-zg_{m-\alpha,0})^T$. 
We {shall denote by  $\mathbf r_n$ quantities} such that $\|\mathbf r_n(z)\|\le \frac {C\tau_n}{v^4}$. \\
Using  these notations we may rewrite the relations (\ref{a6})--({\ref{a9}) as follows, for $\alpha=1,\ldots,m$, 
\begin{equation}\label{b1} 
\mathbf g_{\alpha}= g_{\alpha-1,\alpha-1}\mathbf w_{\alpha}+\mathbf M_t\mathbf L\mathbf g_{\alpha-1}+\mathbf r_n(z),\quad 
\mathbf f_{\alpha}= f_{\alpha-1,\alpha-1}\mathbf y_{\alpha} 
+\mathbf M_u\mathbf L\mathbf f_{\alpha-1}+\mathbf r_n(z), 
\end{equation} 
and 
\begin{equation}\label{b2} 
\mathbf t=-s_n(z){\mathbf f}_{m-1} +\mathbf F\mathbf L\mathbf f_{0}+\mathbf r_n(z),\quad 
\mathbf u=-s_n(z) {\mathbf g}_{m-1} +\mathbf G\mathbf L\mathbf g_{0}+\mathbf r_n(z). 
\end{equation} 
Furthermore, we may represent the relations (\ref{a5}) and (\ref{a7})   as follows 
\begin{equation}\label{b3} 
 \mathbf f_0=-u_0{\mathbf f}_{m-1}+\mathbf F\mathbf L\mathbf u+\mathbf r_n(z),\quad 
\mathbf g_0=-t_0{\mathbf g}_{m-1}+\mathbf G\mathbf L\mathbf t+\mathbf r_n(z). 
\end{equation} 
\begin{lem}\label{k1} 
 Under the conditions of Theorem \ref{main} there exists a sufficiently large
 constant  $V_0$ such that for any $v\ge V_0$ we have 
\begin{equation}\notag 
 \max\{\|\mathbf u\|,\,\|\mathbf t\|,\|\mathbf f_{\alpha}\|, \|\mathbf g_{\alpha}\|\}\le 
\frac {C\tau_n}{v^4}. 
\end{equation} 
\end{lem} 
 
\begin{proof} 
 First we note that, for $z=u+iv$ such that $v>0$ 
\begin{equation}\label{w1} 
 \|\mathbf w_{\alpha}\|+\|\mathbf y_{\alpha}\|\le C(\|\mathbf f_0\|+\|\mathbf g_0\|). 
\end{equation} 
Furthermore, by Lemma \ref{norm1} and inequality  $\|\mathbf R\|\le v^{-1}$, we have 
\begin{equation}\notag 
 \max\{\|\mathbf F\|,\|\mathbf G\|\}\le \frac {C_m}v. 
\end{equation} 
It is straightforward to check that 
\begin{equation}\notag 
 \max\{|zf_{\alpha,\beta}|, |zg_{\alpha,\beta}|\}\le C_m(1+\frac1v). 
\end{equation} 
The last inequalities imply that 
\begin{equation}\notag 
 \max\{\|\mathbf M_u\|,\|\mathbf M_t\|\}\le C_m(1+\frac1v). 
\end{equation} 
Relations (\ref{b2}), (\ref{b3}) together imply that 
\begin{equation}\label{t1} 
 \|\mathbf u\|+\|\mathbf t\|\le \frac {C_m}v(\|\mathbf f_{m-1}\|+\|\mathbf g_{m-1}\|)+\|\mathbf r_n(z)\|. 
\end{equation} 
Relation (\ref{b1}) implies that 
\begin{equation}\notag 
 \|\mathbf g_{m-1}\|+\|\mathbf f_{m-1}\|\le\frac {C_m}v\sum_{q=1}^{m-1}(\|\mathbf w_q\|+\|\mathbf y_q\|) 
+C(\|\mathbf g_0\|+\|\mathbf f_0\|)+\|r_n(z)\|. 
\end{equation} 
Applying now inequality (\ref{w1}), we get 
\begin{equation}\label{t2} 
 \|\mathbf g_{m-1}\|+\|\mathbf f_{m-1}\|\le C_m(\|\mathbf f_0\|+\|\mathbf g_0\|)+\|r_n(z)\|. 
\end{equation} 
Furthermore, relation (\ref{b3}) implies that 
\begin{equation}\label{t31} 
 \|\mathbf f_0\|+\|\mathbf g_0\|\le \frac {C_m}v(\|\mathbf g_{m-1}\|+\|\mathbf f_{m-1}\|)+\frac Cv(\|\mathbf u\|+\|\mathbf t\|)+\|r_n(z)\|. 
\end{equation} 
Inequalities (\ref{t1}), (\ref{t2}), (\ref{t31}) together imply 
\begin{equation}\label{t39} 
 \|\mathbf g_{m-1}\|+\|\mathbf f_{m-1}\|\le \frac {C_m}v(\|\mathbf g_{m-1}\|+\|\mathbf f_{m-1}\|)+\|r_n(z)\|. 
\end{equation} 
Choosing $v_0$ such that $\frac {C_m}v\le \frac14$, we obtain 
\begin{equation}\label{m-1} 
 \|\mathbf g_{m-1}\|+\|\mathbf f_{m-1}\|\le\frac {C_m\tau_n}{v^4}. 
\end{equation} 
Relation (\ref{t1}) implies now that 
\begin{equation}\notag 
 \|\mathbf u\|+\|\mathbf t\|\le \frac {C_m\tau_n}{v^4}. 
\end{equation} 
From relation (\ref{w1}) it follows that 
\begin{equation}\notag 
 \|\mathbf w_{\alpha}\|+\|\mathbf y_{\alpha}\|\le\frac {C\tau_n}{v^4}. 
\end{equation} 
Similar to inequality (\ref{m-1}) we get 
\begin{equation}\notag 
 \|\mathbf g_{\alpha}\|+\|\mathbf f_{\alpha}\|\le\frac {C\tau_n}{v^4}. 
\end{equation} 
Thus the Lemma is proved. 
 
 \end{proof} 
\begin{lem}\label{k2} 
 Under the conditions of Theorem \ref{main} we have 
\begin{equation} 
 f_{\alpha,\alpha}=-zs(z)f_{\alpha-1,\alpha-1}+\varepsilon_n(z). 
\end{equation} 
and 
\begin{equation} 
 g_{\alpha,\alpha}=-zs(z)g_{\alpha-1,\alpha-1}+\varepsilon_n(z). 
\end{equation} 
\end{lem} 
\begin{proof} {We shall consider the first equality only, the proof of the
    other one being similar}. 
By relation (\ref{a4}), we have 
\begin{equation} 
 f_{\alpha,\alpha}=-zf_{0,0}f_{\alpha-1,\alpha-1} 
-z\sum_{q=0}^{\alpha-2}f_{\alpha-1,q}f_{\alpha-1-q,0} 
-\sum_{q=\alpha}^{m-1}f_{\alpha-1,q}u_{m-1+\beta-q}+\varepsilon_n(z). 
\end{equation} 
This equality implies that 
\begin{equation} 
 f_{\alpha,\alpha}=-zs(z)f_{\alpha-1,\alpha-1}+\theta|z|\|\mathbf f_{\alpha-1}\|\|\mathbf f_0\|+\varepsilon_n(z). 
\end{equation} 
Applying Lemma \ref{k1} we conclude the proof. 
\end{proof} 
Equality (\ref{a1}) and Lemma \ref{k2} together imply 
\begin{equation} 
 1+zs_n(z)=-s(z)(f_{m-1,m-1}+g_{m-1,m-1})+\varepsilon_n(z)=(-1)^mz^{m-1}s_n^{m+1}(z)+\varepsilon_n(z) 
\end{equation} 
We rewrite that last equation as follows 
\begin{equation} 
 1+zs_n(z)+(-1)^{m-1}z^{m-1}s_n^{m+1}(z)=\varepsilon_n(z) 
\end{equation} 
The Stieltjes transform $s(z)$ satisfies the equation 
\begin{equation} 
 1+zs(z)+(-1)^{m-1}z^{m-1}s^{m+1}(z)=0 
\end{equation} 
The last two relations together imply that, for $v\ge V_0$ 
\begin{equation} 
 |s_n(z)-s(z)|\le \frac{|\varepsilon_n(z)|}{|z+(-z)^{m-1}\sum_{q=0}^ms^q(z)s_n^{m-q}(z)|} 
\end{equation} 
Note that 
\begin{equation} 
 \max\{|zs(z)|,\,|zs_n(z)|\}\le C(1+\frac1v) 
\end{equation} 
and 
\begin{equation} 
 \max\{|s_n(z)|,\,|s(z)|\}\le \frac1v 
\end{equation} 
Applying these inequality, we obtain 
\begin{equation} 
|(-z)^{m-1}\sum_{q=0}^ms^q(z)s_n^{m-q}(z)| \le \frac Cv. 
\end{equation} 
We may choose $V_1\ge V_0$ such that for any $v\ge V_1$ 
\begin{equation} 
 \frac Cv\le \frac v2. 
\end{equation} 
This implies that for $v\ge V_1$ 
\begin{equation} 
 |\im\{z+(-z)^{m-1}\sum_{q=0}^ms^q(z)s_n^{m-q}(z)\}|\ge \frac v2 
\end{equation} 
and 
\begin{equation}\label{k5} 
|s_n(z)-s(z)|\le \frac{C\tau_n}{v^4}. 
\end{equation} 
From inequality (\ref{k5}) we conclude that there exists 
 an open set with non-empty interior such that $s_n(z)$ converges to $s(z)$ on this set. 
 The Stieltjes transform of these 
random variables is an analytic function on $\mathcal C^+$ 
and locally bounded ($|s_n(z)|\le v^{-1}$ for any $v>0$). 
By Montel's Theorem (see, for instance, 
\cite{Conway}, p. 153, Theorem 2.9) the  convergence of  $s_n(z)$ to $s(z)$
is uniform on  any compact set in the 
upper half-plane $\mathcal K\subset \mathcal C^+$. 
  This implies that $\Delta_n\to 0$ as $n\to\infty$. Thus the proof  of Theorem \ref{main} in 
the general case is complete.

\section{Appendix} 
\subsection{Frobenius norms of powers of truncated matrices} 
Recall that we consider truncated independent random variable $X_{jk}$ satisfying 
\begin{equation}\label{as1} 
 |\E X_{jk}|\le\frac1{n^{\frac32}\tau_n},\quad \E|X_{jk}|^2=1+\frac{\theta_{jk}M}{n^2\tau_n^2}, 
\quad |X_{jk}|\le \tau_n\sqrt n, 
\end{equation} 
where $\tau_n\to 0$  as $n\to\infty$ converges to zero  {as slowly as  needed}. 
 
We would like to investigate   the behavior of the Frobenius norm of powers $\|\mathbf
X^{m}\|_2$
 of the random matrix $\mathbf X=\frac1{\sqrt n}(X_{jk})_{j,k=1}^n$. 
We formulate the following 
\begin{lem}\label{norm0} 
Let $X_{jk}^{(n)}$ be independent random variables for $1\le j,k\le n$ 
and assume that  (\ref{as1}) holds. 
Then for any $m\ge 1$ and any $\nu=0,\ldots , m$, there exists a constant $C_m>0$ depending on $m$ such that 
\begin{equation}\label{norm101} 
\E\|\mathbf X^{\nu}(\mathbf X-\E\mathbf X)^{m-\nu}\|_2^2\le C_mn. 
\end{equation} 
\end{lem} 
\begin{proof} 
We start with the case $\nu=0$. Consider the  matrix
 $\widetilde {\mathbf X}:=\mathbf X-\E \mathbf X=\frac1{\sqrt n}(\widetilde
 X_{jk})_{j,k=1}^n$
 and the norms of the powers of this  matrix. 
We may write 
\begin{align}\label{mom1} 
\E\|\widetilde {\mathbf X}\|_2^2=n^{-m}\sum_{j,k=1}^n\sum_{j_1,\ldots,j_{m-1}=1}^n\sum_{j'_1,\ldots,j'_{m-1}=1}^n&\E\widetilde X_{jj_1}\widetilde X_{j_1j_2}\cdots\widetilde X_{j_{m-2}j_{m-1}}\widetilde X_{j_{m-1}k}\notag\\&\overline{\widetilde X}_{jj_1'}\,\overline{\widetilde X}_{j_1'j_2'}\cdots\overline{\widetilde X}_{j_{m-2}'j_{m-1}'}\,\overline{\widetilde X}_{j_{m-1}'k}. 
\end{align} 
Here for any number $a=u+\sqrt{-1}v$, $\overline a=u-\sqrt{-1}v$ denotes the complex conjugate. 
The product in the right hand side of (\ref{mom1}) {involves  $\mu$
  different } 
(with respect to complex conjugates) terms, say $X^{\varepsilon}_{l_1,l_1'},\ldots, 
X^{\varepsilon}_{l_{\mu},l_{\mu}'}$, with multiplicities $m_1,\ldots, m_{\mu}$, 
where $\varepsilon=\pm$ and 
\begin{equation}X^{\varepsilon}_{jk}=\begin{cases} 
{X_{jk},\quad\text{if}\quad \varepsilon=+}\\{\overline X_{jk},\quad\text{if}\quad \varepsilon=-}\end{cases} 
\end{equation} 
Note that $m_1+\cdots+m_{\mu}=2m$ and if $\min\{m_1,\ldots, m_{\mu}\}=1$ 
then expectation of corresponding product equals 0 
since $\E X^{\varepsilon}_{j,l}=0$ for any $j,l=1,\ldots,n$.
 This implies that {non zero terms occur for} $\mu\le m$ 
and $\min\{m_1,\ldots, m_{\mu}\}\ge 2$ {only}. 
By assumption (\ref{as1}), we have 
\begin{equation}\label{mom2} 
 |\E\widetilde X_{jj_1}\widetilde X_{j_1j_2}\cdots\widetilde X_{j_{m-2}j_{m-1}}\widetilde X_{j_{m-1}k}\overline{\widetilde X}_{jj_1'}\,\overline{\widetilde X}_{j_1'j_2'}\cdots\overline{\widetilde X}_{j_{m-2}'j_{m-1}'}\,\overline{\widetilde X}_{j_{m-1}'k}|\le n^{m-\mu}\tau_n^{2(m-\mu)}. 
\end{equation} 
The cardinality  $\mathcal N(l_1,\ldots,l_m,l_1',\ldots,l_m')$ of  the set of
indices 
with $\mu$ different edges $l_{\nu},l_{\nu}'$  and multiplicities $m_1,\ldots,
m_\mu$ respectively
satisfies the inequality 
\begin{equation}\label{mom3} 
\mathcal N(l_1,\ldots,l_m,l_1',\ldots,l_m')\le Cn^{\mu+1}. 
\end{equation} 
The representation (\ref{mom1}) and the inequalities (\ref{mom2} and (\ref{mom3}) together 
imply 
\begin{equation}\label{mom100} 
 \E\|{\widetilde{\mathbf X}}^m\|_2^2\le C_mn 
\end{equation} 
Assume  now that $1\le \nu\le m$. 
Consider the quantity 
\begin{equation} 
 \Gamma_n^{(\nu)}=\E\|\mathbf X^{\nu}(\mathbf X-\E\mathbf X)^{m-\nu}\|_2^2. 
\end{equation} 
Let $\mathbf A=\E\mathbf X$. 
It is straightforward to check that 
\begin{equation}\label{mom5} 
 \Gamma_n^{(\nu)}\le C_m\sum_{\alpha=0}^{\nu}(\E\|\mathbf A\|_2^2)^{\alpha}\E\|(\mathbf X-\mathbf A)^{m-\alpha}\|_2^2. 
\end{equation} 
To prove (\ref{mom5}) we consider the representation 
\begin{equation} 
 \mathbf X^{\nu}={\sum}^*\mathbf A^{m_1}(\mathbf X-\mathbf A)^{m_1'} 
\cdots\mathbf A^{m_{\nu}}(\mathbf X-\mathbf A)^{m_{\nu}'}, 
\end{equation} 
where $\sum^*$ stands for sum  over all indices $m_1,\ldots, m_k,m_1',\ldots,m_{\nu}'\ge0$ 
such that \newline$ 
m_1+\cdots+m_{\nu}+m_1'+\cdots+m_{\nu}'=\nu$. 
This implies the bound 
\begin{equation}\label{mom6} 
 \Gamma_n^{(\nu)}\le C_m{\sum}^* 
\E\|\mathbf A^{m_1}(\mathbf X-\mathbf A)^{m_1'} 
\cdots\mathbf A^{m_{\nu}}(\mathbf X-\mathbf A)^{m_{\nu}'}\|_2^2. 
\end{equation} 
Using that for any matrices $\mathbf A$ and $\mathbf B$ we have $\|\mathbf A\mathbf B\|_2= 
\|\mathbf B\mathbf A\|_2$ and that $\|\mathbf A^{\nu}\|_2\le \|\mathbf
A\|_2^{\nu}$,
 we get from  this inequality the bound (\ref{mom5}). 
By assumption (\ref{as1}), we have 
\begin{equation}\label{mom7} 
 \|\mathbf A\|_2^2\le \frac{CM}{n^2\tau_n^6}. 
\end{equation} 
Inequalities (\ref{mom6}), (\ref{mom7}), 
(\ref{mom100}) and the induction assumption together conclude the proof of the
Lemma.

\end{proof} 
 
} 
We shall use the following obvious  bounds 
\begin{lem}\label{norm1} 
Let $X_{jk}^{(n)}$ be independent random variables for $1\le j,k\le n$. 
assume that  (\ref{as1}) holds and that  $\E X_{jk}=0$. 
Then for any $m,r\ge1$ and any $\nu=0,\ldots , m$, and any $j=1,\ldots,n$,
 there exists a constant $C(m,r)>0$ depending on $m,\,r$ such that 
\begin{equation}\label{norm102} 
\max\{\E\|\mathbf X^{\nu}\mathbf e_j\|_2^{2r},\,\E\|\mathbf e_j^T\mathbf X^{\nu}\|_2^{2r}\}\le C(m,r). 
\end{equation} 
\end{lem} 
\begin{proof} 
Let 
\begin{equation} 
 \Gamma_{\nu,j}=\|\mathbf X^{\nu}\mathbf e_j\|_2. 
\end{equation} 
We may write 
\begin{align} 
\Gamma_{\nu,j}^2=\sum_{k_1,\ldots,k_r=1}^n\sum_{j_1^{(1)},\ldots,j_{\nu-1}^{(1)}=1}^n& 
\sum_{{j_1^{(1)}}',\ldots,{j_{\nu-1}^{(1)}}'=1}^n\cdots\sum_{j_1^{(r)},\ldots,j_{\nu-1}^{(r)}=1}^n\notag\\&\times\sum_{{j_1^{(r)}}',\ldots,{j_{\nu-1}^{(r)}}'=1}^n 
\prod_{q=1}^r A(k_q,j_1^{(q)},\ldots,j_{\nu-1}^{(q)})\overline A(k_q,{j_1^{(q)}}',\ldots,{j_{\nu-1}^{(q)}}'), 
\end{align} 
where 
\begin{equation}A(k_q,j_1^{(q)},\ldots,j_{\nu-1}^{(q)})= 
X_{k_qj_1^{(q)}}X_{j_1^{(q)}j_2^{(1)}}\cdots X_{j_{\nu-2}^{(q)}j_{\nu-1}^{(q)}}X_{j_{\nu-1}^{(q)}j}. 
\end{equation} 
Assume that  the set of indices 
$\mathcal N=\cup_{q=1}^r \{\{k_q,j_1^{(q)},\ldots,j_{\nu-1}^{(q)}\}\cup{j_1^{(q)}}',\ldots,{j_{\nu-1}^{(q)}}'\}$ 
 {consists of}  $\mu$ different 
 {pairs}, say ${l_1,l_1'},\ldots, 
{l_{\mu},l_{\mu}'}$, with multiplicities $m_1,\ldots, m_{\mu}$ respectively. 
Note that $m_1+\cdots+m_{\mu}=2qr$ and if $\min\{m_1,\ldots, m_{\mu}\}=1$ then 
the  corresponding term  equals 0,  
since $\E X_{j,l}=0$ for any $j,l=1,\ldots,n$. This implies that $\mu\le m$ 
and $\min\{m_1,\ldots, m_{\mu}\}\ge 2$. 
By assumption (\ref{as1}), we have 
\begin{equation}\label{mom20} 
 |\E A(k_q,j_1^{(q)},\ldots,j_{\nu-1}^{(q)})\overline A(k_q,{j_1^{(q)}}',\ldots,{j_{\nu-1}^{(q)}}'|\le C(\tau_n\sqrt n)^{2mr-2\mu}. 
\end{equation} 
The cardinality  $\mathcal N(l_1,\ldots,l_m,l_1',\ldots,l_m')$ of the
 set of indices with $\mu$ different edges $l_{\nu},l_{\nu}$ 
and multiplicities $m_1,\ldots, m_\mu$ satisfies the inequality 
\begin{equation}\label{mom30} 
\mathcal N(l_1,\ldots,l_m,l_1',\ldots,l_m')\le Cn^{\mu}. 
\end{equation} 
The representation (\ref{mom1}) and the inequalities (\ref{mom20} and (\ref{mom30}) together 
imply 
\begin{equation}\label{mom1001} 
 \E\|\mathbf X^{\nu}\mathbf e_j\|_2^{2r}\le C(m,r) 
\end{equation} 
The bound of  $\E\|\mathbf e_j^T\mathbf X^{\nu}\|_2^{2r}$ is  similar. 
Thus, lemma is proved. 
\end{proof}

\begin{lem}\label{var0}Under the conditions of Theorem \ref{main} we have 
\begin{equation} 
\E|\frac1n(\Tr \mathbf R-\E\Tr \mathbf R)|^2\le \frac C{nv^2}. 
\end{equation} 
\end{lem} 
\begin{proof}Consider the matrix $\mathbf X^{(j)}$ obtained from the matrix $\mathbf 
X$  by replacing the  entries of the  $j$-th row by zeros. 
We define the following matrices 
\begin{equation} \mathbf H^{(j)}=\left(\begin{matrix}{\mathbf X^{(j)}\quad 
\mathbf O}\\{\mathbf O\quad {\mathbf X^{(j)}}^*}\end{matrix}\right), \qquad 
\widehat {\mathbf H}^{(j)}=\mathbf H^{(j)}\mathbf J. 
\end{equation} 
We shall use the following inequality.
 For any Hermitian matrix $\mathbf A$ and $\mathbf B$ 
with spectral distribution function $F_A(x)$ and $F_B(x)$ respectively, 
we have 
\begin{equation}\label{trace} 
|\Tr (\mathbf A-z\mathbf I)^{-1}-\Tr (\mathbf B-z\mathbf I)^{-1}|\le 
\frac {\text{\rm rank}(\mathbf A-\mathbf B)}{v}. 
\end{equation} 
It is straightforward to show that 
\begin{equation}\label{rank} 
\text{\rm rank}(\mathbf H^q{\mathbf J}-{\mathbf H^{(j)}}^q\mathbf J)\le 4q. 
\end{equation} 
Inequalities  (\ref{trace}) and (\ref{rank}) together imply 
\begin{equation} 
|\frac1{2n}(\Tr \mathbf R-\Tr \mathbf R^{(j)})|\le \frac C{nv}. 
\end{equation} 
After this remark 
we may apply a well-known  martingale expansion techniques suggested  already by Girko 
\cite{Girko:89}. We may introduce 
$\sigma$-algebras $\mathcal F_{j}=\sigma\{X_{lk},\, 
j< l\le n, k=1,\ldots,n\}$ and use the representation 
\begin{equation}\notag 
\Tr\mathbf R-\E\Tr\mathbf R=\sum_{\nu=1}^m\sum_{j=1}^{n}(\E_{j-1}\Tr\mathbf R-\E_{j}\Tr\mathbf R), 
\end{equation} 
where $\E_{j}$ denotes conditional  expectation given  $\sigma$-algebra $\mathcal F_{j}$. 
\end{proof} 
\begin{lem}\label{var1} 
Under the  conditions of Theorem \ref{main} we have, for $q\ge 1$ 
\begin{equation} 
\E|\frac1n(\sum_{j=1}^n[\mathbf H^q\mathbf J\mathbf R]_{jj+n}-\E\sum_{j=1}^n[\mathbf H^q\mathbf J\mathbf R]_{jj+n})|^2\le \frac C{n v^4}. 
\end{equation} 
\end{lem} 
\begin{proof} 
We introduce the matrices 
$\mathbf X^{(j)}=\mathbf X-\mathbf e_j\mathbf e_j^T\mathbf X$, and $\mathbf H^{(j)}=\mathbf H-\mathbf e_j\mathbf e_j^T\mathbf H-\mathbf H\mathbf e_{j+n}\mathbf e_{j+n}^T$. 
Note that the matrix $\mathbf X^{(j)}$ is obtained 
from the matrix $\mathbf X$ by replacing the entries of the $j$-th row by $0$. 
Consider the quantity 
\begin{equation} 
S_j:=\sum_{k=1}^n[\mathbf H^q\mathbf J\mathbf R]_{kk+n}-\sum_{k=1}^n[{\mathbf H^{(j)}}^q\mathbf J\mathbf R^{(j)}]_{kk+n}. 
\end{equation} 
Using equality 
\begin{align} 
\mathbf H^q\mathbf J\mathbf R-{\mathbf H^{(j)}}^q\mathbf J\mathbf R^{(j)}&=\sum_{\nu=0}^{q-1}{\mathbf H^{(j)}}^{\nu}(\mathbf H-\mathbf H^{(j)}){\mathbf H}^{q-1-\nu} 
\mathbf J\mathbf R 
\notag\\&+\sum_{\nu=0}^{p-1}{\mathbf H^{(j)}}^q\mathbf J\mathbf R^{(j)}\mathbf H^{\nu}(\mathbf H-\mathbf H^{(j)})\mathbf H^{p-1-\nu}\mathbf J\mathbf R, 
\end{align} 
we get 
\begin{align} 
S_j=S_j^{(1)}+S_j^{(2)}, 
\end{align} 
where 
\begin{align} 
S_j^{(1)}&=\sum_{\nu=0}^{q-1}\sum_{k=1}^n[{\mathbf H^{(j)}}^{\nu}(\mathbf H-\mathbf H^{(j)}){\mathbf H}^{q-1-\nu} 
\mathbf J\mathbf R]_{kk+n} 
\notag\\S_j^{(2)}=&\sum_{\nu=0}^{p-1}\sum_{k=1}^n[{\mathbf H^{(j)}}^q\mathbf J\mathbf R^{(j)}\mathbf H^{\nu}(\mathbf H-\mathbf H^{(j)})\mathbf H^{p-1-\nu}\mathbf J\mathbf R]_{kk+n}. 
\end{align} 
Applying now that 
\begin{equation} 
\mathbf H-\mathbf H^{(j)}=\mathbf e_j\mathbf e_j^T\mathbf H+\mathbf H\mathbf e_{j+n}\mathbf e_{j+n}^T, 
\end{equation} 
we obtain 
\begin{align}\label{m1} 
\sum_{k=1}^n[{\mathbf H^{(j)}}^{\nu}(\mathbf H-\mathbf H^{(j)}){\mathbf H}^{q-1-\nu} 
\mathbf J\mathbf R]_{kk+n}&=\Tr\widehat{\mathbf J}{\mathbf H^{(j)}}^{\nu}(\mathbf e_j\mathbf e_j^T\mathbf H+\mathbf H\mathbf e_{j+n}\mathbf e_{j+n}^T){\mathbf H}^{q-1-\nu} 
\mathbf J\mathbf R\notag\\ 
=[{\mathbf H}^{q-\nu}& 
\mathbf J\mathbf R\mathbf J{\mathbf H^{(j)}}^{\nu}]_{jj}+[{\mathbf H}^{q-1-\nu} 
\mathbf J\mathbf R\mathbf J{\mathbf H^{(j)}}^{\nu}\mathbf H]_{j+nj+n}|. 
\end{align} 
Here 
\begin{equation} 
 \widehat J=\left(\begin{matrix}{\mathbf O\qquad\mathbf I}\\{\mathbf O\qquad\mathbf O}\end{matrix}\right). 
\end{equation} 
Equality (\ref{m1}) implies that 
\begin{equation} 
|S_j^{(1)}|\le \sum_{\nu=0}^{q-1}|[{\mathbf H}^{q-\nu} 
\mathbf J\mathbf R\mathbf J{\mathbf H^{(j)}}^{\nu}]_{jj}|+|[{\mathbf H}^{q-1-\nu} 
\mathbf J\mathbf R\mathbf J{\mathbf H^{(j)}}^{\nu}\mathbf H]_{j+nj+n}|. 
\end{equation} 
Using  H\"older's  inequality, we get 
\begin{equation} 
\E |S_{j}^{(1)}|^2\le\frac{C}{v^2}\sum_{\nu=0}^{q-1}(\E\|\mathbf e_j^T\mathbf H^{q-\nu}\|_2^{2}\|{\mathbf H^{(j)}}^{\nu} 
\mathbf e_j\|_2^2+\E\|\mathbf e_{j+n}{\mathbf H}^{q-1-\nu}\|_2^2\|{\mathbf H^{(j)}}^{\nu}\mathbf H\mathbf e_{j+n}\|_2^2) 
\end{equation} 
Lemma \ref{norm1} and H\"older's inequality together imply 
\begin{equation}\label{m2} 
\E |S_j^{(1)}|\le \frac C{v^2}. 
\end{equation} 
Similar we get 
\begin{equation}\label{m3} 
\E|S_j^{(2)}|^2\le \frac {C}{v^4}. 
\end{equation} 
Inequalities (\ref{m2}) and (\ref{m3}) together imply 
\begin{equation} 
 \E|S_j|^2\le \frac {C(v^2+1)}{v^4}. 
\end{equation} 
 
Let $\mathcal F_j$ denote the $\sigma$-algebra generated by $X_{lk}$, for $1\le l\le j, 1\le k\le n$. 
Denote by $\E_j$ the conditional  expectation with respect to $\sigma$-algebra $\mathcal F_j$. 
We may write 
\begin{align} 
\E|\frac1n(\sum_{j=1}^n&[\mathbf H^q\mathbf J\mathbf R]_{jj+n}-\E\sum_{j=1}^n[\mathbf H^q\mathbf J\mathbf R]_{jj+n})|^2\notag\\&= 
\frac1{n^2}\sum_{j=1}^n\E|\E_j\sum_{k=1}^n[\mathbf H^q\mathbf J\mathbf R]_{kk+n}-\E_{j-1}\sum_{k=1}^n[\mathbf H^q\mathbf J\mathbf R]_{kk+n})|^2\notag\\&\le 
\frac1{n^2}\sum_{j=1}^n\E|\sum_{k=1}^n[\mathbf H^q\mathbf J\mathbf R]_{kk+n}-\sum_{k=1}^n[{\mathbf H^{(j)}}^q\mathbf J\mathbf R^{(j)}]_{kk+n})|^2\notag\\&\le 
\frac{C(1+v^2)}{ nv^4}. 
\end{align} 
Thus the Lemma is proved. 
\end{proof} 
\begin{lem}\label{var2}Under the conditions of Theorem \ref{main} the following inequality holds 
\begin{equation} 
\E|\frac1n(\sum_{j=1}^n\mathbf R_{j,j+n}-\E\sum_{j=1}^n\mathbf R_{j,j+n})|^2\le \frac C{nv^4}. 
\end{equation} 
\end{lem} 
\begin{proof} 
The proof is similar to the proof of the previous lemma. 
We have 
\begin{equation} 
\sum_{k=1}^n\mathbf R_{kk+n}-\sum_{k=1}^n\mathbf R^{(j)}_{kk+n}=\sum_{\nu=0}^{p-1}\sum_{k=1}^n[\mathbf R^{(j)}{\mathbf H^{(j)}}^{\nu}(\mathbf e_j\mathbf e_j^T\mathbf H+\mathbf H\mathbf e_{j+n}\mathbf e_{j+n}^T)\mathbf H^{p-1-\nu}\mathbf R]_{kk+n}. 
\end{equation} 
Applying H\"older's inequality and  inequality $\max\{\|\mathbf R\|,\|\mathbf R^{(j)}\|\}\le v^{-1}$, we get 
\begin{align} 
|\sum_{k=1}^n\mathbf R_{kk+n}-\sum_{k=1}^n\mathbf R^{(j)}_{kk+n}|&\le 
\frac1{v^2}\sum_{\nu=0}^{p-1} \|{\mathbf H^{(j)}}^{\nu}\mathbf e_j\|_2 
\|\mathbf e_{j}^T\mathbf H^{p-\nu}\|_2\notag\\&+\frac1{v^2}\sum_{\nu=0}^{p-1} \|\mathbf e_{j+n}^T\mathbf H^{p-1-\nu}\|_2\|{\mathbf H^{(j)}}^{\nu}\mathbf H\mathbf e_{j+n}\|_2. 
\end{align} 
Using H\"older inequality and Lemma \ref{norm1}, we get 
\begin{align} 
\E|\sum_{k=1}^n\mathbf R_{kk+n}-\sum_{k=1}^n\mathbf R^{(j)}_{kk+n}|^2 
\le \frac{C_m}{v^4}. 
\end{align} 
To conclude the proof it is enough to use the martingale expansion of
the {difference}
\newline $\sum_{k=1}^nR_{kk+n}-\sum_{k=1}^n\E R_{kk+n}$ similar to previous lemma. 
 
\end{proof} 

\begin{lem}\label{derivatives} 
Under the conditions of Theorem \ref{main} we have, for $0\le \mu, \nu\le m$,
that there exists a constant  $C_m$ depending on $m$ such that 
\begin{equation} 
\left|n^{-\frac32}\sum_{j,k=1}^n\E (X_{jk}+X_{jk}^3)\left[\frac{\partial ^{2}(\mathbf H^{\nu}\mathbf J\mathbf R\mathbf H^{\mu})}{\partial X_{jk}^{2}}(\theta_{jk}X_{jk})\right]_{kj}\right| 
\le \frac{C_m(1+v)}{\sqrt n v^{3}}, 
\end{equation} 
where $\theta_{jk}$ and $X_{jk}$ are mutually independent   $j,k=1,\ldots,n$,  and 
$\theta_{jk}$ are uniformly distributed on the unit interval. 
By $\frac{\partial^2}{\partial {X_{jk}}^2}\mathbf A(\theta_{jk}X_{jk})$ we
denote 
the matrix obtained from 
$\frac{\partial^2}{\partial {X_{jk}}^2}\mathbf A$
 by replacing the  entries 
$X_{jk}$ by $\theta_{jk}X_{jk}$. 
 
\end{lem} 
\begin{proof}By the formula for  derivatives of a resolvent matrix , we have 
\begin{align} 
\frac{\partial (\mathbf H^{\nu}\mathbf J\mathbf R\mathbf H^{\mu})}{\partial X_{jk}}&=\frac1{\sqrt n}\sum_{a=0}^{\nu-1}\mathbf H^{a}(\mathbf e_j\mathbf e_k^T+ 
\mathbf e_{k+n}\mathbf e_{j+n}^T)\mathbf H^{\nu-1-a}\mathbf J\mathbf R\mathbf H^{\mu}\notag\\&-\frac1{\sqrt n}\sum_{b=0}^{m-1} 
\mathbf H^{\nu}\mathbf J\mathbf R\mathbf H^{b}(\mathbf e_j\mathbf e_k^T+ 
\mathbf e_{k+n}\mathbf e_{j+n}^T)\mathbf H^{m-1-b}\mathbf J\mathbf R\mathbf H^{(\mu)}\notag\\& 
+\sum_{c=0}^{\mu-1}\mathbf H^{\nu}\mathbf J\mathbf R\mathbf H^{c}(\mathbf e_j\mathbf e_k^T+ 
\mathbf e_{k+n}\mathbf e_{j+n}^T)\mathbf H^{\mu-1-c}. 
\end{align} 
From this formula it follows that
\begin{align} 
\frac{\partial^2 (\mathbf H^{\nu}\mathbf J\mathbf R\mathbf H^{\mu})}{\partial X_{jk}^2}& 
=\frac1{n}(\sum_{a=0}^{\nu-1}\sum_{i=1}^{4}\mathbf P_{i}^{(a)}+ 
\sum_{c=1}^{\mu-1}\sum_{i=1}^{4}\mathbf T_i^{(c)})-\frac1{n}\sum_{b=1}^{m-1}\sum_{i=1}^6 
\mathbf U_{i}^{(b)}, 
\end{align} 
where 
\begin{align} 
\mathbf P_1^{(a)}=-&\sum_{s=0}^{a-1}\mathbf H^{s}(\mathbf e_j\mathbf e_k^T+ 
\mathbf e_{k+n}\mathbf e_{j+n}^T)\mathbf H^{a-1-s}(\mathbf e_j\mathbf e_k^T+ 
\mathbf e_{k+n}\mathbf e_{j+n}^T)\mathbf H^{\nu-1-a}\mathbf J\mathbf R\mathbf H^{\mu}\notag\\ 
\mathbf P_2^{(a)}=-&\frac1{n}\sum_{s=0}^{\nu-a-1} 
\mathbf H^{a}(\mathbf e_j\mathbf e_k^T+ 
\mathbf e_{k+n}\mathbf e_{j+n}^T)\mathbf H^{s}(\mathbf e_j\mathbf e_k^T+ 
\mathbf e_{k+n}\mathbf e_{j+n}^T)\mathbf H^{\nu-1-a-s}\mathbf J\mathbf R\mathbf H^{\mu}\notag\\ 
\mathbf P_3^{(a)}=-&\frac1{n}\sum_{s=0}^{m-1} 
\mathbf H^{a}(\mathbf e_j\mathbf e_k^T+ 
\mathbf e_{k+n}\mathbf e_{j+n}^T)\mathbf H^{\nu-1-s}\mathbf J\mathbf R\mathbf H^{s}(\mathbf e_j\mathbf e_k^T+ 
\mathbf e_{k+n}\mathbf e_{j+n}^T)\mathbf H^{m-1-s}\mathbf J\mathbf R\mathbf H^{\mu}\notag\\ 
\mathbf P_4^{(a)}=-&\frac1{n}\sum_{s=0}^{\mu-1}\mathbf H^{a}(\mathbf e_j\mathbf e_k^T+ 
\mathbf e_{k+n}\mathbf e_{j+n}^T)\mathbf H^{\nu-1-s}\mathbf J\mathbf R 
\mathbf H^{s}(\mathbf e_j\mathbf e_k^T+ 
\mathbf e_{k+n}\mathbf e_{j+n}^T)\mathbf H^{\mu-1-s}.\notag 
\end{align} 
Furthermore, 
\begin{align} 
\mathbf T_1^{(c)}=&-\frac1{n}\sum_{s=0}^{\nu-1} 
\mathbf H^{s}(\mathbf e_j\mathbf e_k^T+ 
\mathbf e_{k+n}\mathbf e_{j+n}^T)\mathbf H^{\nu-1-s}\mathbf J\mathbf R\mathbf H^{c}(\mathbf e_j\mathbf e_k^T+ 
\mathbf e_{k+n}\mathbf e_{j+n}^T)\mathbf H^{\mu-1-c}\notag\\ 
\mathbf T^{(c)}_2=&-\frac1{n}\sum_{s=0}^{m-1}\mathbf H^{\nu}\mathbf J\mathbf R\mathbf H^{s}(\mathbf e_j\mathbf e_k^T+ 
\mathbf e_{k+n}\mathbf e_{j+n}^T)\mathbf H^{m-1-s}\mathbf J\mathbf R\mathbf H^c(\mathbf e_j\mathbf e_k^T+ 
\mathbf e_{k+n}\mathbf e_{j+n}^T)\mathbf H^{\mu-1-c}, 
\notag\\ 
\mathbf T^{(c)}_3=&-\frac1{n}\sum_{s=0}^{c-1} 
\mathbf H^{\nu}\mathbf J\mathbf R\mathbf H^{s}(\mathbf e_j\mathbf e_k^T+ 
\mathbf e_{k+n}\mathbf e_{j+n}^T)\mathbf H^{c-1-s}(\mathbf e_j\mathbf e_k^T+ 
\mathbf e_{k+n}\mathbf e_{j+n}^T)\mathbf H^{\mu-1-c}\notag\\ 
\mathbf T^{(c)}_4=&-\sum_{s=1}^{\mu-2-c}\mathbf H^{\nu}\mathbf J\mathbf R\mathbf H^c(\mathbf e_j\mathbf e_k^T+ 
\mathbf e_{k+n}\mathbf e_{j+n}^T)\mathbf H^s(\mathbf e_j\mathbf e_k^T+ 
\mathbf e_{k+n}\mathbf e_{j+n}^T)\mathbf H^{\mu-2-c-s}. 
\end{align} 
Finally, 
\begin{align} 
 \mathbf U^{(b)}_1&=- \sum_{s=0}^{\nu-1}\mathbf H^s(\mathbf e_j\mathbf e_k^T+ 
\mathbf e_{k+n}\mathbf e_{j+n}^T)\mathbf H^{\nu-1-s}\mathbf J\mathbf R\mathbf H^b(\mathbf e_j\mathbf e_k^T+ 
\mathbf e_{k+n}\mathbf e_{j+n}^T)\mathbf H^{m-1-b}\mathbf J\mathbf R\mathbf H^{\mu}\notag\\ 
\mathbf U^{(b)}_2&=- \sum_{s=0}^{m-1}\mathbf H^{\nu}\mathbf J\mathbf R\mathbf H^{s}(\mathbf e_j\mathbf e_k^T+ 
\mathbf e_{k+n}\mathbf e_{j+n}^T)\mathbf H^{m-1-s}\mathbf J\mathbf R\mathbf H^b(\mathbf e_j\mathbf e_k^T+ 
\mathbf e_{k+n}\mathbf e_{j+n}^T)\mathbf H^{m-1-b}\mathbf J\mathbf R\mathbf H^{\mu}\notag\\ 
\mathbf U^{(b)}_3&=- \sum_{s=0}^{b-1}\mathbf H^{\nu}\mathbf J\mathbf R\mathbf H^{s}(\mathbf e_j\mathbf e_k^T+ 
\mathbf e_{k+n}\mathbf e_{j+n}^T)\mathbf H^{b-1-s}(\mathbf e_j\mathbf e_k^T+ 
\mathbf e_{k+n}\mathbf e_{j+n}^T) 
\mathbf H^{m-1-b}\mathbf J\mathbf R\mathbf H^{\mu}\notag\\ 
\mathbf U^{(b)}_4&=- \sum_{s=0}^{b-1}\mathbf H^{\nu}\mathbf J\mathbf R\mathbf H^{b}(\mathbf e_j\mathbf e_k^T+ 
\mathbf e_{k+n}\mathbf e_{j+n}^T)\mathbf H^{s}(\mathbf e_j\mathbf e_k^T+ 
\mathbf e_{k+n}\mathbf e_{j+n}^T) 
\mathbf H^{m-2-b}\mathbf J\mathbf R\mathbf H^{\mu}\notag\\ 
\mathbf U^{(b)}_5&=- \sum_{s=0}^{m-1}\mathbf H^{\nu}\mathbf J\mathbf R\mathbf H^{b}(\mathbf e_j\mathbf e_k^T+ 
\mathbf e_{k+n}\mathbf e_{j+n}^T)\mathbf H^{m-1-b}\mathbf J\mathbf R\mathbf H^s(\mathbf e_j\mathbf e_k^T+ 
\mathbf e_{k+n}\mathbf e_{j+n}^T) 
\mathbf H^{m-1-s}\mathbf J\mathbf R\mathbf H^{\mu}\notag\\ 
\mathbf U^{(b)}_6&=- \sum_{s=0}^{\mu-1}\mathbf H^{\nu}\mathbf J\mathbf R\mathbf H^{b}(\mathbf e_j\mathbf e_k^T+\mathbf e_{k+n}\mathbf e_{j+n}^T)\mathbf H^{m-1-b}\mathbf J\mathbf R\mathbf H^s(\mathbf e_j\mathbf e_k^T+ 
\mathbf e_{k+n}\mathbf e_{j+n}^T) 
\mathbf H^{\mu-1-s}.\notag 
\end{align} 
 
Note that for any matrices $\mathbf A$ and $\mathbf B$ we have 
\begin{equation}\label{in1} 
 |[\mathbf A\mathbf B]_{jk}|\le \|\mathbf e_k^T\mathbf A\|_2\|\mathbf A\mathbf e_j\|_2. 
\end{equation} 
Applying H\"older's  inequality, we get, for $\alpha=1$ or $\alpha=3$ 
\begin{equation} 
 \E|X_{jk}|^{\alpha}|[\mathbf P^{(a)}]_{kj}|\le \E^{\frac{\alpha}4}|X_{jk}|^4\E^{\frac{4-\alpha}4}|[\mathbf P^{(a)}]_{kj}|^{\frac4{4-\alpha}}. 
\end{equation} 
We may use now 
 inequality (\ref{in1}) and Lemma \ref{norm1} to obtain the bound, 
for $\alpha=1$ or $\alpha=3$ 
\begin{equation}\label{in2} 
 \frac1{n^{\frac52}}\sum_{j,k=1}^n\E|X_{jk}|^{\alpha}|[\mathbf P^{(a)}_l]_{kj}|\le \frac {C(1+v^2)}{\sqrt n v^3}. 
\end{equation} 
Similar we get 
\begin{equation}\label{in3} 
 \frac1{n^{\frac52}}\sum_{j,k=1}^n\E|X_{jk}|^3|[\mathbf T^{(c)}_l]_{kj}|\le \frac {C(1+v)}{\sqrt n v^3} 
\end{equation} 
and 
\begin{equation}\label{in4} 
 \frac1{n^{\frac52}}\sum_{j,k=1}^n\E|X_{jk}|^3|[\mathbf U^{(b)}_l]_{kj}|\le \frac {C(1+v)}{\sqrt n v^3} 
\end{equation} 
 Inequalities (\ref{in2})--(\ref{in4}) together conclude the proof of the Lemma. 
\end{proof}

\begin{lem}\label{teilor} 
Under conditions of Theorem \ref{main} we have, for $\mu, \nu\ge 0$ and 
for any positive $\eta>0$ that there exists a constant  $C(\mu, \nu,\eta)$ 
depending on $\mu,\nu,\eta$ such that 
\begin{equation} 
\sum_{j,k=1}^n\E X_{jk}[\mathbf H^{\nu}\mathbf J\mathbf R\mathbf H^{\mu}]_{kl} 
=\sum_{j,k=1}^n\E[\frac{\partial (\mathbf H^{\nu}\mathbf J\mathbf R\mathbf H^{\mu})}{\partial X_{jk}}]_{kl}+\frac{C\theta(1+v}{\sqrt nv^3}, 
\end{equation} 
where $\theta$ denotes a function that $|\theta|\le 1$. 
\end{lem} 
\begin{proof}Let $\xi$ be random variable with $\E\xi=0$, $E\xi^2=1$ 
and let $f(x)$ denote a function which satisfies the following condition 
$\E|\xi|^3|f"(\theta\xi)|\le \varkappa$. 
Here  $\theta$ denotes a uniformly  distributed random variable  on $[0,1]$.
 By Tailor's formula we have 
\begin{equation} 
\E\xi f(\xi)=\E f'(\xi)-\E\xi f''(\theta\xi)+\frac12\E\xi^3 f''(\theta\xi), 
\end{equation} 
where $\theta$ denotes a uniformly  distributed random variable independent of $\xi$. 
Applying this formula twice and H\"older's  inequality, we get 
\begin{align} 
|\sum_{j,k=1}^n\E X_{jk}[\mathbf H^{\nu}\mathbf J\mathbf R\mathbf H^{\mu}]_{kj}-\sum_{j,k=1}^n\E[\frac{\partial (\mathbf H^{\nu}\mathbf J\mathbf R\mathbf H^{\mu})} 
{\partial X_{jk}}]_{kj}|\notag\\ \le n^{-\frac32}\sum_{j,k=1}^n\E(|X_{jk}|+|X_{jk}|^3)\left|\left[\frac{\partial ^2(\mathbf H^{\nu}\mathbf J\mathbf R\mathbf H^{\mu})} 
{\partial X_{jk}^2}(\theta_{jk}X_{jk})\right]_{kj}\right| 
\end{align} 
Applying now  the result of Lemma \ref{derivatives}, we conclude the proof of Lemma. 
\end{proof} 
 
 

\begin{thebibliography}{99} 
\itemsep=\smallskipamount 
 

\bibitem{AGT}Alexeev, N.; G\"otze, F.; Tikhomirov, A. N.
 {\em On the asymptotic distribution of singular values of power of  random matrices.},
 Lithuanian mathematical journal, Vol. 50, No. 2, 2010, pp. 121--132.
\bibitem{AGT:2010a}Alexeev, N.; G\"otze, F.; Tikhomirov, A. N.
 {\em On the singular spectrum of powers and products of random matrices},
{ Doklady mathematics}, vol. 82, N 1, 2010, pp.505--507.



 
\bibitem{capitaine:08} Banica, T. Belinschi, S. Capitaine,   M. and Collins B. 
{\em Free Bessel Laws} 
Preprint. arXiv:0710.5931 
 
\bibitem{Conway}Conway, John B. 
{\em Functions of one complex variable I}.\newline  Springer--Verlag, 
Berlin 1995 - 2nd ed., 316 pp. 
 
\bibitem{Girko:89} Girko, V. L. {\em Spectral theory of random matrices}. (Russian) 
Uspekhi Mat. Nauk  {\bf 40}  (1985),  no. 1(241), 67--106. 
 
\bibitem{Oravecz}Oravecz F.{\em On the powers of Voiculescu's circular element}.  Studia Math.  {\bf 145}  (2001),  no. 1, 85--95. 
 
\bibitem{Speicher}Mingo, J. A. and Speicher, R. 
{\em Sharp Bounds for Sums Associated to Graphs of Matrices} 
Preprint. arXiv:0909.4277 


 
 
 
\end{thebibliography}
\end{document}